\documentclass{amsart}

\usepackage{amsmath,amsthm,amsfonts, amssymb}
\usepackage{graphicx}

\newtheorem{theorem}{Theorem}[section]
\newtheorem{proposition}[theorem]{Proposition}
\newtheorem{lemma}[theorem]{Lemma}
\newtheorem{corollary}[theorem]{Corollary}

\theoremstyle{definition}
\newtheorem{definition}[theorem]{Definition}
\newtheorem*{remark}{Remark}

\newcommand{\C}{\mathbb C}
\newcommand{\R}{\mathbb R}

\newcommand{\Z}{\mathbb Z}
\newcommand{\N}{\mathbb N}
\newcommand{\ovl}[1]{\overline{#1}}
\newcommand{\zbar}{\ovl z}
\newcommand{\cbar}{\ovl c}
\renewcommand{\Re}{\operatorname{Re}}
\renewcommand{\Im}{\operatorname{Im}}
\newcommand{\eps}{\varepsilon}
\renewcommand{\phi}{\varphi}
\newcommand{\hide}[1]{}
\newcommand{\M}{\mathcal {M}^*}
\newcommand{\Md}{\M_d}
\newcommand{\sm}{\setminus}
\newcommand{\Cstar}{\C^*}

\newcommand{\Includegraphics}[2]{\includegraphics[#1]{#2}}

\newcommand{\re}{\Re}

\renewcommand{\marginpar}[1]{}

\title{Multicorns are not Path Connected}
\author{John Hubbard}
\author{Dierk Schleicher}

\dedicatory{Dedicated to John Milnor on the occasion of his 80th
birthday, \\in gratitude for much inspiration, friendship, and
generosity: mathematical and otherwise.}

\begin{document}

\renewcommand{\bottomfraction}{.5}

\maketitle

\bigskip
\begin{center}
Milnor, that intrepid explorer,\\
Traveled cubics in hopes to discover\\
\quad Some exotic new beast:\\
\quad North-west and south-east \\
He found tricorns lurking there under cover.
\end{center}

\section{Introduction}

The \emph{multicorn} $\Md$ is the connectedness locus in the space of 
antiholomorphic unicritical polynomials $p_c(z)=\zbar^d+c$ of degree 
$d$, i.e., the set of parameters for which the Julia set is 
connected. The special case $d=2$ is the \emph{tricorn}, which is the 
formal antiholomorphic analog to the Mandelbrot set.

\begin{figure}
\framebox{
\Includegraphics{height=60mm}{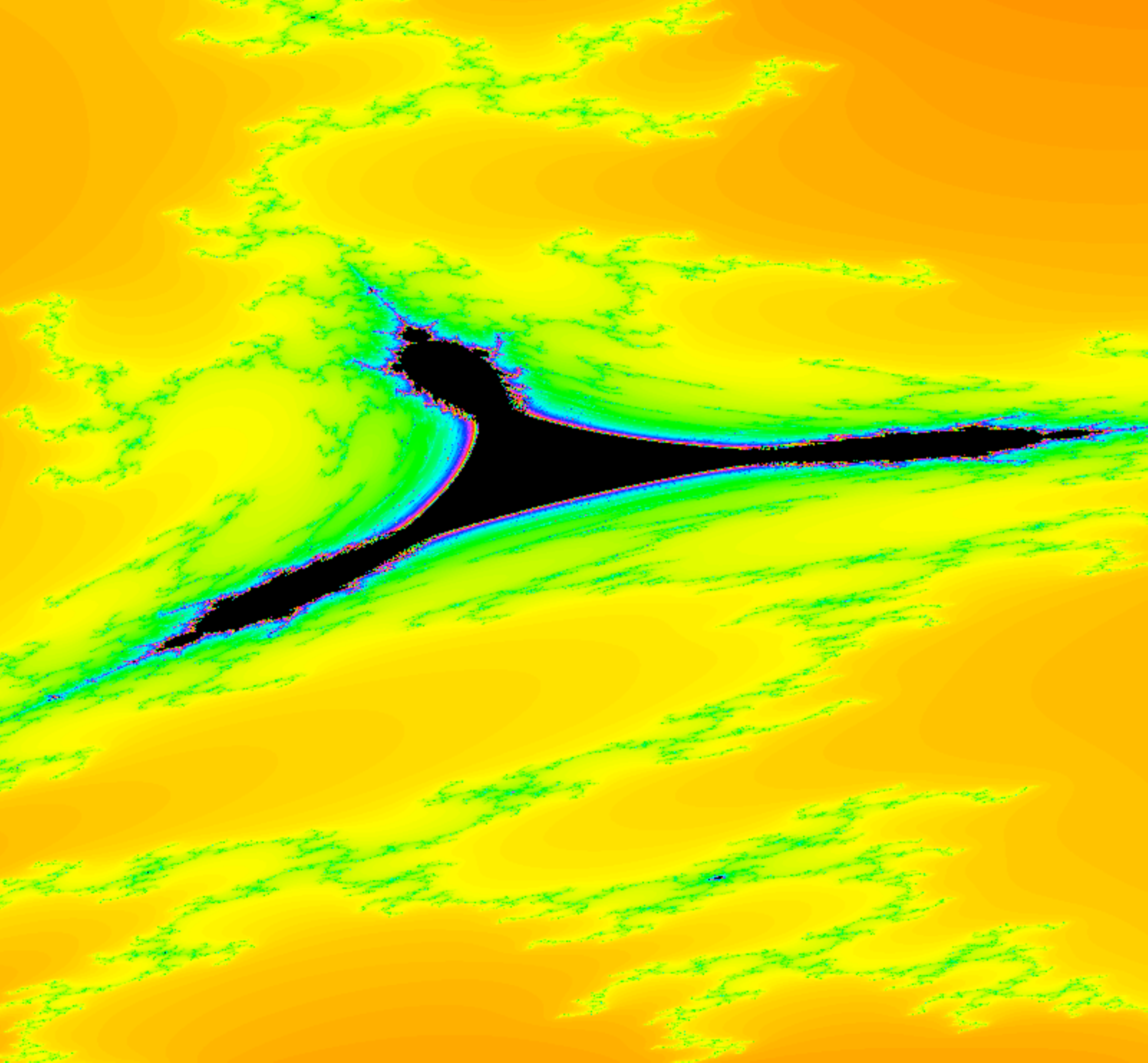}
}
\caption{A ``little tricorn'' within the tricorn $\mathcal M_2^*$ illustrating that the ``umbilical cord'' converges to the little tricorn without landing at it. }
\end{figure}

The second iterate is
\[
p_c^{\circ 2}(z)=\ovl{ (\ovl z^{d}+c)}^d+c=(z^d+\ovl c)^d+c
\]
and thus holomorphic in the dynamical variable $z$, but no longer 
complex analytic in the parameter $c$. Much of the dynamical theory 
of antiholomorphic polynomials (in short, antipolynomials) is thus in 
analogy to the theory of holomorphic polynomials, except for certain 
features near periodic points of odd periods. For instances, a 
periodic point of odd period $k$ may be the simultaneous landing 
point of dynamic rays of periods $k$ and $2k$ (which is invisible 
from the holomorphic second iterate of the first return map); see 
\cite[Lemma~3.1]{Multicorns1}.

However, the theory of parameter space of multicorns is quite 
different from that of the Mandelbrot set and its higher degree 
cousins, the multibrot sets of degree $d$, because the parameter 
dependence is only real analytic. Already the open mapping principle 
of the multiplier map fails, so it is not a priori clear that every 
indifferent orbit is on the boundary of a hyperbolic component, and 
that bifurcations multiplying periods occur densely on boundaries of 
hyperbolic components. However, it turns out that many properties of 
parameter space are quite similar to that of the Mandelbrot set, 
except near hyperbolic components of odd period. For instance, there 
is a simple recursive relation for the number of hyperbolic 
components of period $n$ for the multibrot set, given by 
$s_{d,n}=d^{n-1}-\sum_{k|n,\, k<n} s_{d,k}$: for multicorns, the same 
result holds, except if $n$ is twice an odd number; in that case, the 
number of hyperbolic components equals $s_{d,n}+2s_{d,n/2}$ 
\cite{Multicorns2}. Similarly, the multiplier map is an open map on 
the closure of any hyperbolic component of even period, except where 
it intersects the boundary of an odd period hyperbolic component.

However, boundaries of odd period hyperbolic components have some 
quite interesting properties. The multiplier map is constant along 
their boundaries (all boundary points have parabolic orbits of 
multiplier $+1$); bifurcations only double the 
period (no higher factors), and these period-doublings occur along arcs rather than at 
isolated points (see Corollary~\ref{Cor:BifArcs}). Adjacent to these 
parabolic arcs, there are comb-like 
structures where the multicorn fails local connectivity, and  $\sin 1/x$-like structures accumulate on the 
centers of many parabolic arcs:  even 
pathwise connectivity fails there. Nonetheless, some boundary arcs of odd 
period hyperbolic components also feature ``open beaches'' with 
sub-arcs of positive length that form part of boundary of a 
hyperbolic component without any further decorations, so the 
hyperbolic component and the escape locus meet along a smooth arc. 
(We do not know whether the number of such arcs is finite or not.)

\medskip
\emph{Overview of Paper and Results}.
In this paper we study the boundaries of hyperbolic components of $\Md$ of odd period, focussing on local connectivity and pathwise connectivity. In 
Section~\ref{Sec:ParabolicDynamics}, we investigate parabolic 
dynamics especially of odd period, review Ecalle 
cylinders and their special features in antiholomorphic dynamics: the 
existence of an invariant curve called the {\it equator}. We then discuss 
parabolic arcs on the boundary of hyperbolic components of odd period. 
In Section~\ref{Sec:FixedPointIndex}, we investigate these arcs from 
the point of view of the holomorphic fixed point index, and we show 
that period-doubling bifurcations occur near both ends of all 
parabolic arcs. We then discuss, in Section~\ref{Sec:ParabolicPerturbations}, perturbations of parabolic periodic points and introduce continuous coordinates for the perturbed dynamics.
In Section~\ref{Sec:Combinatorics} we introduce an 
invariant tree in parabolic dynamics, similar to the Hubbard 
tree for postcritically finite polynomials, and discuss the dynamical 
properties of parabolic maps that we will later transfer into 
parameter space. This transfer is then done in 
Section~\ref{Sec:Non-Pathwise}, using perturbed Fatou coordinates: 
these are somewhat simplified in the antiholomorphic setting because 
of the existence of the invariant equator. Some concluding remarks 
and further results are discussed in Section~\ref{Sec:FinalSection}.

\medskip
\emph{Relations to Holomorphic Parameter Spaces}.
Much of the relevance of tricorn (and the higher dimensional 
multicorns) comes from the fact that it is related to natural 
holomorphic parameter spaces. Clearly, the tricorn space is the (real 
two-dimensional but not complex-analytic) slice $c=\ovl a= b$ in the 
complex two-dimensional space of maps $z\mapsto (z^2+a)^2+b$, one of 
the natural complex two-dimensional spaces of polynomials. Perhaps 
more interestingly, the tricorn is naturally related to the space of 
\emph{real} cubic polynomials: this space can be parametrized as 
$z\mapsto \pm z^3-3a^2z +b$ with $a,b\in\R$. It was in the context of 
this space  that Milnor discovered and explored the tricorn 
\cite{MilnorCubics,MilnorHypComps} as one of the prototypical local 
dynamical features in the presence of two active critical points; compare Figure~\ref{Fig:RealCubics}. To 
see how antiholomorphic dynamics occurs naturally in the dynamics of 
a real cubic polynomial $p$, suppose there is an open bounded 
topological disk $U\subset\C$ containing one critical point so that 
$p(U)$ contains the closure of the complex conjugate of $U$. Denoting 
complex conjugation of $p$ by $p^*$ and the closure of $U$ by $\ovl 
U$, we have $p^*(U)\supset \ovl U$. If, possibly by suitable 
restriction, the map $p\colon U\to p(U)$ is proper holomorphic, then 
$p^*\colon U\to p^*(U)$  is the antiholomorphic analogue of a 
polynomial-like map in the sense of Douady and Hubbard. Since $p$ 
commutes with complex conjugation, the dynamics of $p$ and of $p^*$ 
are the same (the even iterates coincide), so the dynamics of $p$ 
near one critical point is naturally described by the antiholomorphic 
polynomial $p^*$ (and the other critical point is related by 
conjugation). The advantage of the antiholomorphic point of view is 
that, while $U$ and $p(U)$ may be disjoint domains in $\C$ without 
obvious dynamical relation, there is a well-defined antiholomorphic 
dynamical system $p^*\colon U\to p^*(U)$. This is even more useful 
when $\ovl U$ is not a subset of $p^*(U)$, but of a higher iterate: 
in this case, like for ordinary polynomial-like maps, the interesting 
dynamics of a high degree polynomial is captured by a low-degree 
polynomial or, in this case, antipolynomial.

\begin{figure}
\Includegraphics{height=47mm}{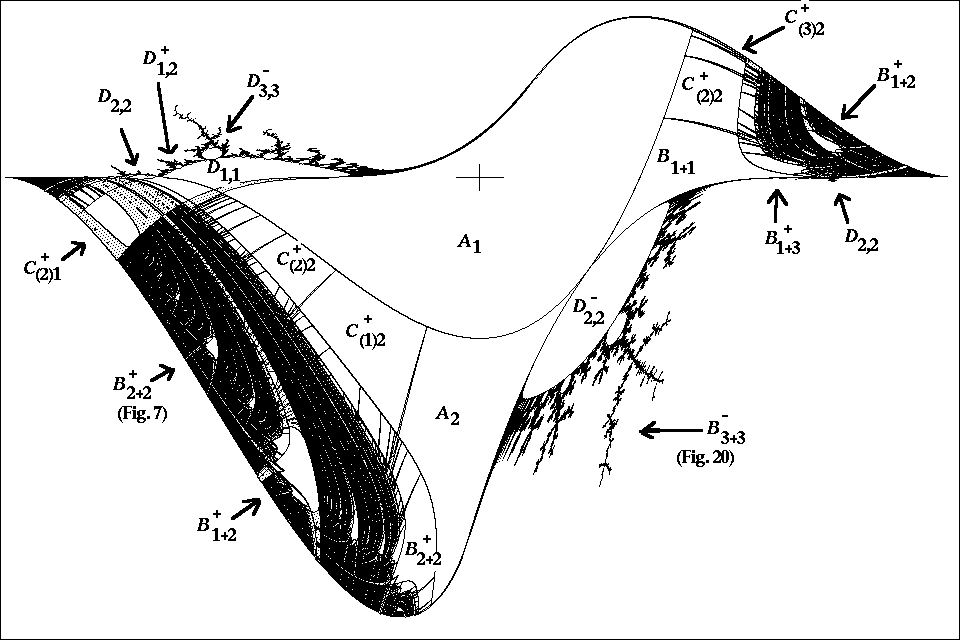}
\Includegraphics{height=47mm}{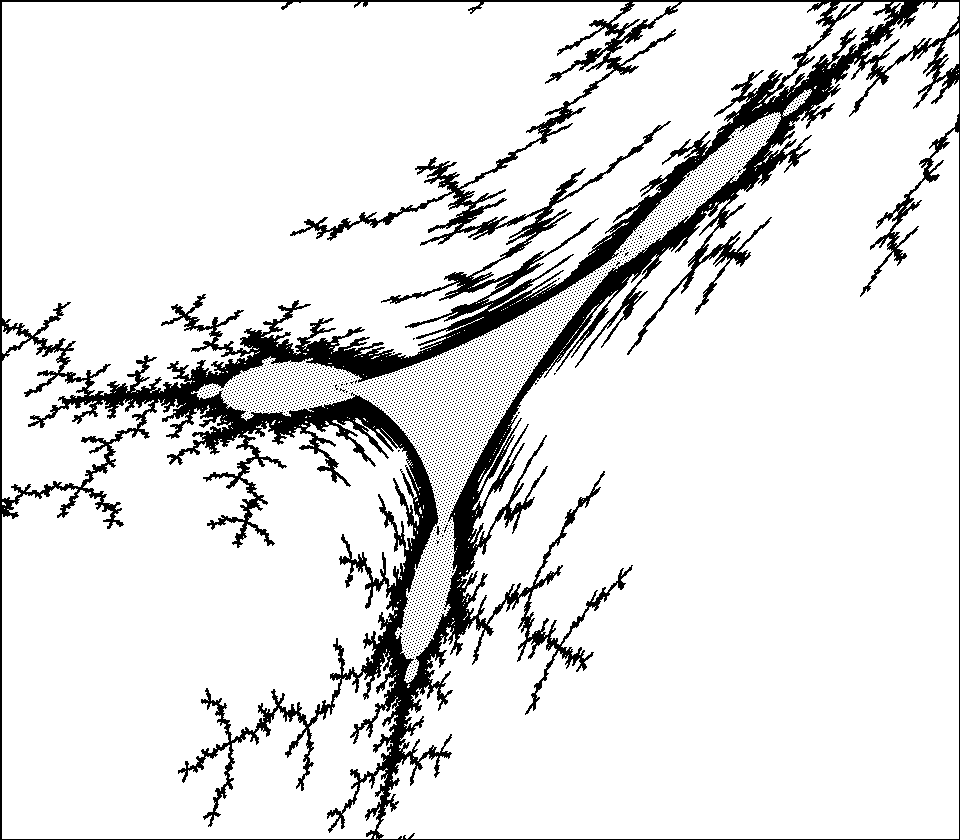}
\caption{The connectedness locus of real cubic polynomials and a detail from the south-east quadrant, showing a tricorn-like structure. (Pictures from Milnor~\cite{MilnorCubics}.)}
\label{Fig:RealCubics} 
\end{figure}

\medskip
\emph{Are There Embedded Tricorns?}
It was numerically ``observed'' by several people that the tricorn 
contains, around each hyperbolic component of even period, a small 
copy of the Mandelbrot set, and around each odd period component a 
small copy of the tricorn itself; and similar statements hold for 
certain regions of the real cubic connectedness locus --- much as the 
well known fact that the Mandelbrot set contains a small copy of 
itself around each hyperbolic component. A small tricorn within the 
big one is shown in Figure~\ref{Fig:BigAndLittleTricorn}. However, we 
believe that most, if not all, ``little tricorns'' are not 
homeomorphic to the actual tricorn (both within the tricorn space and 
within the real cubic locus); quite possibly most little tricorns 
might not even be homeomorphic to each other. Indeed, a subset of the 
real axis connects the main hyperbolic component (of period $1$) to 
the period $3$ ``airplane'' component (along the real axis, the 
tricorn and the Mandelbrot set coincide obviously): we say that the 
``umbilical cord'' of the period $3$ tricorn lands. However, we prove 
for many little tricorns that their umbilical cords do not land but 
rather forms some kind of $\sin 1/x$-structure. Our methods only 
apply to ``prime'' little tricorns: these are the ones not contained 
in larger ``little tricorns'', so we do not disprove continuity of 
the empirically observed embedding map given by the straightening 
theorem (even though this seems very likely). Two little tricorns 
could only be homeomorphic to each other if they have matching sizes 
of the wiggles of the umbilical cords of all the infinitely many 
little tricorns they contain, where the size of such a wiggle is 
measured in terms of Ecalle heights as introduced below.

Failure of continuity of the straightening map was shown in other 
contexts, for instance by Epstein and by Inou. Failure of local 
connectivity and of pathwise connectivity was numerically observed by 
Milnor \cite{MilnorCubics} for the tricorn. For complex parameter 
spaces, failure of local connectivity was observed by Lavaurs for the 
cubic connectedness locus (a brief remark in his thesis) and by 
Epstein and Yampolsky~\cite{EpsteinYampolsky} for real slices of 
cubic polynomials. Nakane and Komori \cite{NakaneKomori} showed that 
certain ``stretching rays'' in the space of real cubic polynomials do 
not land.

\begin{remark}
This work was inspired by John Milnor in many ways: he was the first
to have observed the tricorn and its relevance in the space of
iterated (real) cubic maps, he made systematic studies about the
local behavior of parameter spaces and under which conditions little
tricorns appear there, he observed the loss of local connectivity and
even of path connectivity of the tricorn, he introduced the term
``tricorn'' --- and his home page shows non-landing rays of the kind
that he observed and that we discuss here.
\end{remark}


\begin{figure}
\Includegraphics{height=57mm,trim=13 0 0 0,clip}{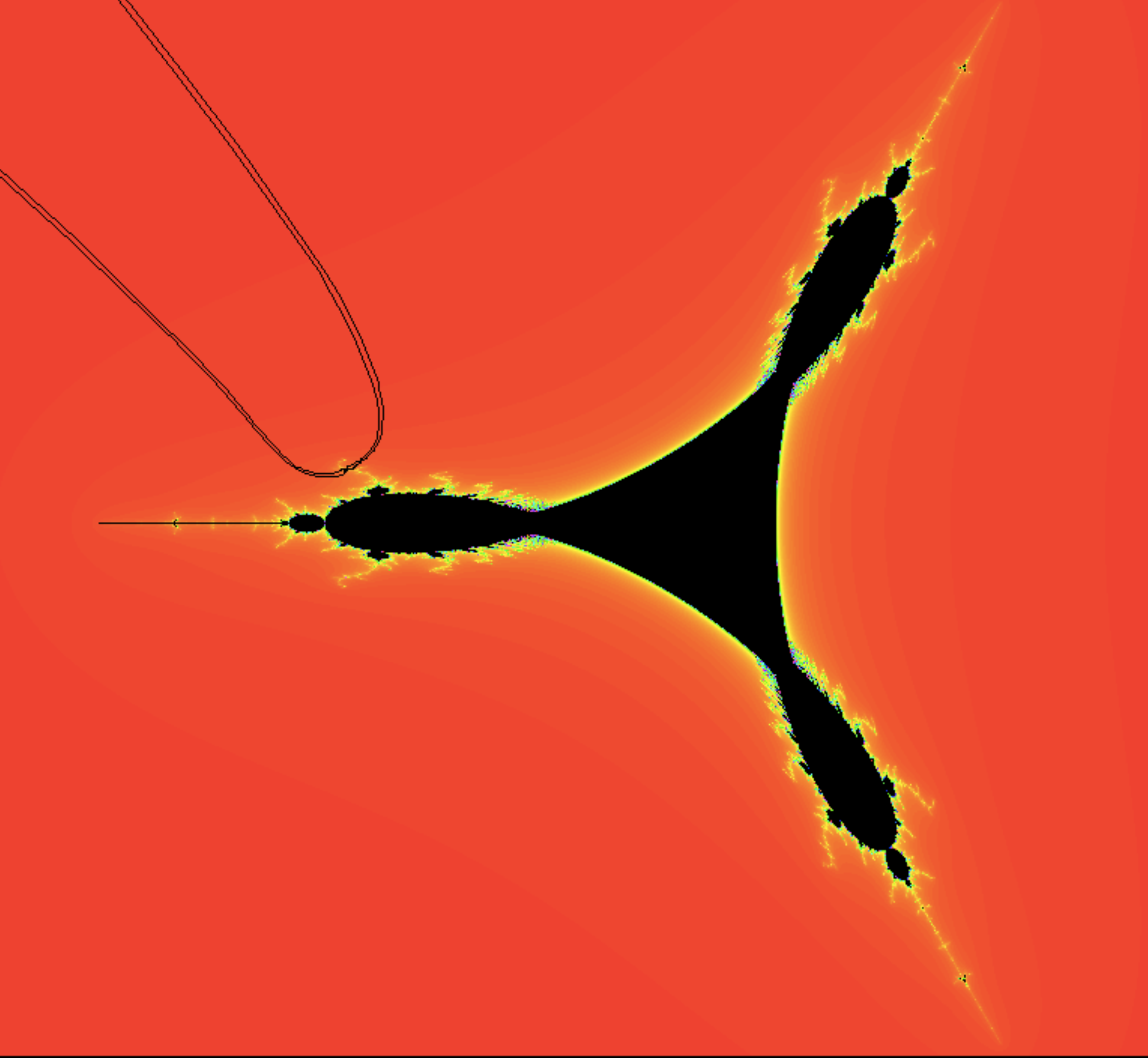}
\Includegraphics{height=57mm,trim=13 0 0 0,clip}{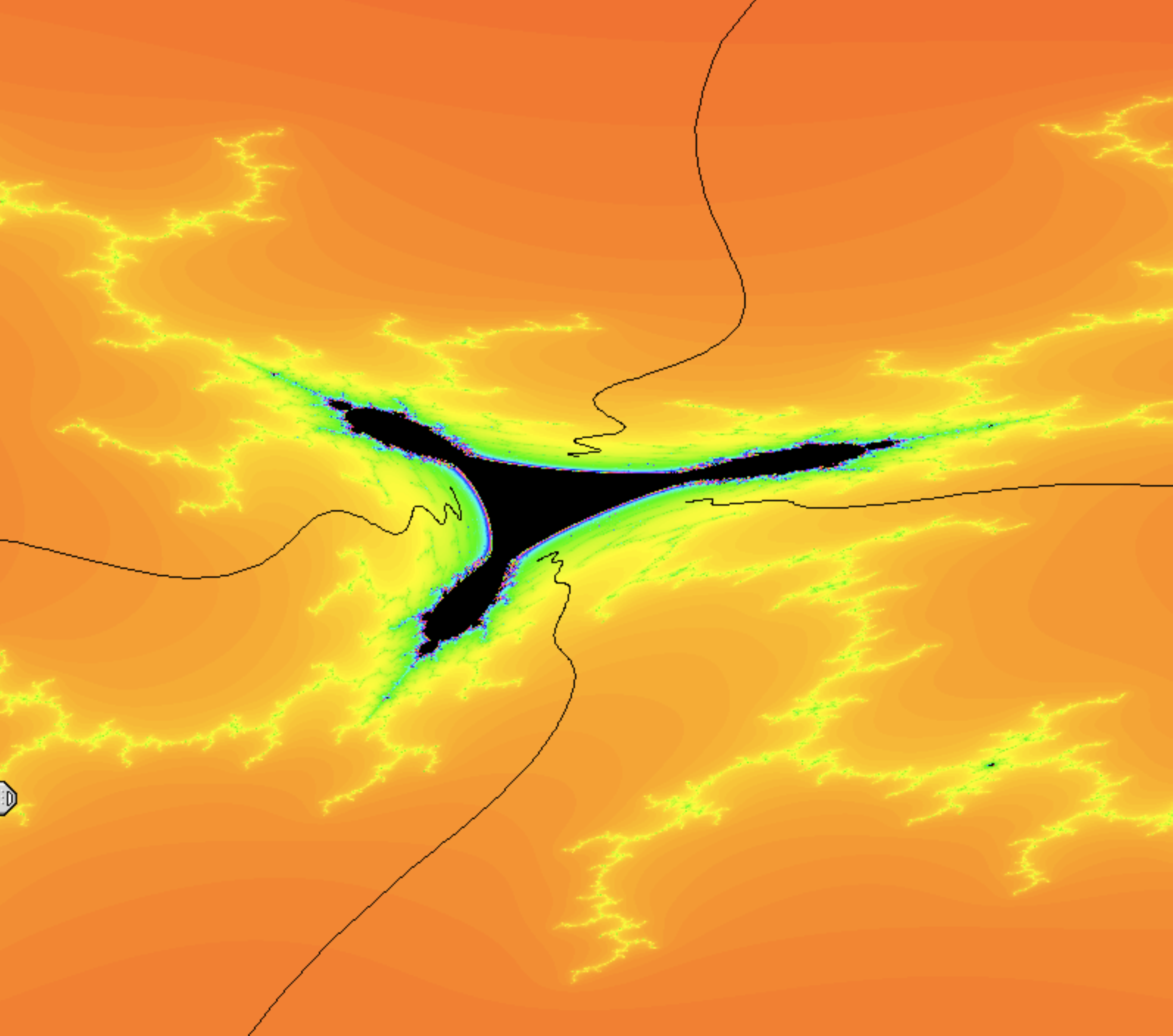}
\caption{The tricorn and a blow-up showing a ``small tricorn'' of 
period $5$. Shown in both pictures are the four parameter rays 
accumulating at the boundary of the period $5$ hyperbolic component 
(at angles $371/1023$, $12/33$, $13/33$, and $1004/1023$). The wiggly 
features of these non-landing rays are clearly visible in the 
blow-up.}
\label{Fig:BigAndLittleTricorn}
\end{figure}

\begin{remark}
In this paper, we need certain background results from the (still 
unpublished) earlier manuscript \cite{Multicorns2} which has more 
detailed results of bifurcations especially at hyperbolic components 
of odd period. In setting up notation and background, it seems more 
convenient to complete these proofs here rather than to strictly 
avoid overlap.
\end{remark}

\emph{Acknowledgements.}
We would like to thank Adam Epstein for many inspiring and helpful discussions on tricorns, parabolic perturbations, and more. We would also like to thank Shizuo Nakane for numerous discussions, many years ago, about antiholomorphic dynamics. We are most grateful to two anonymous referees for numerous detailed and helpful comments. Finally, D.S.\ would like to thank Cornell University for its hospitality, and the German Research Council DFG for its support during the time this work was carried out.

\section{Antiholomorphic and Parabolic Dynamics}
\label{Sec:ParabolicDynamics}

In many ways, antiholomorphic maps have similar dynamical properties as holomorphic ones because the second iterate is holomorphic. There are a number of interesting features specific to antiholomorphic dynamics though, especially near periodic points of odd period $k$. The multiplier of a periodic point of odd period $k$ is not a conformal invariant; only its absolute value is, and the multiplier of the $2k$-th iterate (the second return map) is always non-negative real. This has interesting consequences on boundaries of hyperbolic components of odd period: all boundary parameters are parabolic with multiplier $+1$ (for the holomorphic second return map). 

Another unusual feature is that dynamic rays landing at the same point of odd period $k$ need not all have the same period. These rays can have period $k$ or $2k$ (not higher), and both periods of rays can coexist: see \cite[Lemma~3.1]{Multicorns1}.

We will also show that the number of periodic points of odd period $k$ can change: but of course the number of periodic points of periods $k$ and $2k$, which are both periodic points of period $k$ for the holomorphic second iterate, must remain constant; the only thing that can happen is that two orbits of odd period $k$ turn into one orbit of period $2k$, and this always occurs on boundaries of hyperbolic components of odd period $k$: see Lemma~\ref{Lem:PeriodDoubling} below.

The Straightening Theorem \cite{Polylike} for polynomials has an antiholomorphic analogue. We state it here for easier reference; the proof is the same as in the holomorphic case.
\begin{theorem}[The Antiholomorphic Straightening Theorem]
\label{Thm:Straightening}
Suppose that $U\subset V\subset\C $ are two bounded topological disks so that the closure of $U$ is contained in $V$. Suppose also that $f\colon U\to V$ is an antiholomorphic proper map of degree $d$. Then $f|_U$ is hybrid equivalent to an antiholomorphic polynomial $p$ of the same degree $d$. If the filled-in Julia set of $f\colon U\to V$ (the set of points that can be iterated infinitely often) is connected, then $p$ is unique up to conformal conjugation. 
\qed
\end{theorem}
As usual, two maps are hybrid equivalent if they are quasiconformally conjugation so that the complex dilatation vanishes on the filled-in Julia set.

In the rest of this section, we discuss the local dynamics of parabolic periodic points of odd period $k$ specifically for antipolynomials $p_c(z)=\zbar^d+c$.

\begin{lemma}[Simple and Double Parabolics]
Every parabolic periodic point of $p_c$ of odd period, when viewed as 
a fixed point of an even period iterate of $p_c$, has parabolic 
multiplicity $1$ or $2$.
\end{lemma}
\begin{proof}
The first return map of any parabolic periodic point of odd period is 
antiholomorphic, but the second iterate of the first return map is 
holomorphic and has multiplier $+1$. 
This second iterate can thus be written in local coordinates as 
$z\mapsto z+z^{q+1}+\dots$, where $q\ge 1$ is  the 
multiplicity of the parabolic orbit. There are then $q$ attracting Fatou 
petals, and each must absorb an infinite critical orbit of 
$p_c^{\circ 2}$. But $p_c^{\circ 2}$ has two critical orbits (the 
single critical orbit of $p_c$ splits up into two orbits of 
$p_c^{\circ 2}$, for even and odd iterates), hence $q\le 
2$. (Viewing this periodic point as a fixed point of a higher iterate 
of $p_c^{\circ 2}$ does not change $q$: in the same local coordinates 
as before, the higher iterate takes the form $z\mapsto 
z+az^{q+1}+\dots$, where $a\in\N$ measures which higher iterate we 
are considering.)
\end{proof}

A parabolic periodic point with multiplicity $1$ (resp.\ $2$) is called a
\emph{simple (resp.\ double) parabolic point}.
A parameter $c$ so that $p_c$ has a double parabolic periodic point 
is called a \emph{parabolic cusp}.

\begin{lemma}[Ecalle cylinders]
\label{LemEcalleCylinders}
Let $z_0$ be a simple parabolic periodic point of odd period $k$ of an
antiholomorphic map $f$ and let $V$ be the attracting basin of $z_0$.
Then there is a neighborhood $U$ of $z_0$ and an analytic map 
$\phi:U\cap V \to \C$  that is an isomorphism to the half-plane $\Re 
w>0$ such that
\[
\phi\circ  f^{\circ k}\circ \phi^{-1}(w)=\overline w+1/2.
\]
The map $\phi$ is unique up to an additive real constant.
\end{lemma}

It follows that the quotient  of $V\cap U$ by $f^{\circ 2k}$ is 
isomorphic to $\C/\Z$, and on this quotient cylinder $f$ induces the 
map $x+iy\mapsto x+1/2-iy$ with $x\in \R/\Z,\ y\in \R$.

\begin{proof}
The second iterate $f^{\circ 2}$ is holomorphic, and for this map 
$z_0$ is parabolic with period $k$. Since the parabolic point is 
simple, we have the usual conformal Fatou coordinates  $\phi\colon 
V\cap U\to\C$ with $\phi\circ f^{\circ 2k}\circ\phi^{-1}(w)=w+1$ for 
a certain neighborhood $U$ of $z_0$, where $\phi(V\cap U)$ covers 
some right half plane and $\phi$ is unique up to addition of a 
complex constant. Adjusting this constant and restricting $U$ (which will no longer be a neighborhood of $z_0$), we may 
assume that $\phi(V\cap U)$ is the right half plane $\Re w>0$. It 
follows that $(V\cap U)/(f^{\circ 2k})$ is conformally isomorphic to 
the bi-infinite $\C/\Z\simeq\Cstar$, so that $f^{\circ k}$ has to 
send this cylinder to itself in an antiholomorphic way. The only 
antiholomorphic automorphisms of $\C/\Z$ are $w\mapsto \pm \ovl 
w+\alpha'$ with $\alpha'\in\C/\Z$ (depending on the sign, the two ends 
of $\C/\Z$ are either fixed or interchanged), and lifting this to the 
right half plane we get $\phi\circ f^{\circ k}\circ \phi^{-1}(w)=\pm 
\ovl w+\alpha$ with $\alpha\in\C$, hence $\phi\circ f^{\circ 
2k}\circ\phi^{-1}(w)=w+\pm\ovl\alpha+\alpha\stackrel{!}{=}w+1$, so 
either $2\Re\alpha=1$ or $2i\Im\alpha=1$. The latter case is 
impossible, and in the former case we get $\Re\alpha=1/2$ as claimed. 
But $\phi$ is still unique up to addition of a complex constant, and 
the imaginary part of this constant can be adjusted uniquely so that 
$\alpha$ becomes real, i.e., $\alpha=1/2$.
\end{proof}

\begin{definition}[Ecalle cylinder, Ecalle height, and equator]
The quotient cylinder $(V\cap U)/(f^{\circ 2k})$  isomorphic to
$\C/\Z$ is called the \emph{Ecalle cylinder} of the attracting basin.
Its \emph{equator} is the unique simple closed (Euclidean) geodesic
of $\C/\Z$ that is fixed by the action of $f$: in those coordinates
in which $f$ takes the form $w\mapsto \ovl w+1/2$, this equator is the
projection of $\R$ to the quotient. Finally, the \emph{Ecalle height} 
of a point $w\in\C/\Z$ in the Ecalle cylinder is defined as $\Im w$.

Similarly to the Ecalle cylinders in the attracting basin, one can
also define them for a local branch of $f^{-1}$ fixing $z_0$; all 
this requires is the local parabolic dynamics in a neighborhood of 
$z_0$. To distinguish
these, they are called \emph{incoming} and \emph{outgoing} Ecalle
cylinders (for $f$ and $f^{-1}$, respectively). Both have equators 
and Ecalle heights.
\end{definition}

Note that the identification of an Ecalle cylinder with $\C/\Z$ for 
usual holomorphic maps is  unique only up to translation by a complex 
constant; in our case with an antiholomorphic intermediate iterate 
and thus a
preferred equator, this identification is unique up to a real
constant. Therefore, there is no intrinsic meaning of $\Re w$ within 
the cylinder, or for $\Re\phi(z)$ for $z\in V\cap U$.
However, for two points $z,z'\in V\cap U$, the difference
$\Re \phi(z)-\Re \phi(z')$ has a well-defined meaning in $\R$ called
\emph{phase difference}; this notion actually extends to the entire 
attracting basin $V$.

\begin{proposition}[Parabolic Arcs]
\label{Prop:ParabolicArc}
Every polynomial  $p_c$ with a simple parabolic periodic point of odd
period is part of a real $1$-dimensional family of parabolic maps
$p_{c(h)}$ with simple parabolic orbits. This family is real
analytically parametrized by Ecalle height $h$ of the critical value; 
more precisely, the map $h\mapsto p_{c(h)}$ is a real-analytic 
bijection from $\R$ onto a family of parabolic maps that we call a 
\emph{parabolic arc}.
\end{proposition}

We sketch the proof in Figure~\ref{Fig:ChangeEcalleHeight}; see 
\cite[Theorem~3.2]{Multicorns2} for details.

\begin{figure}[h]
\Includegraphics{width=.7\textwidth}{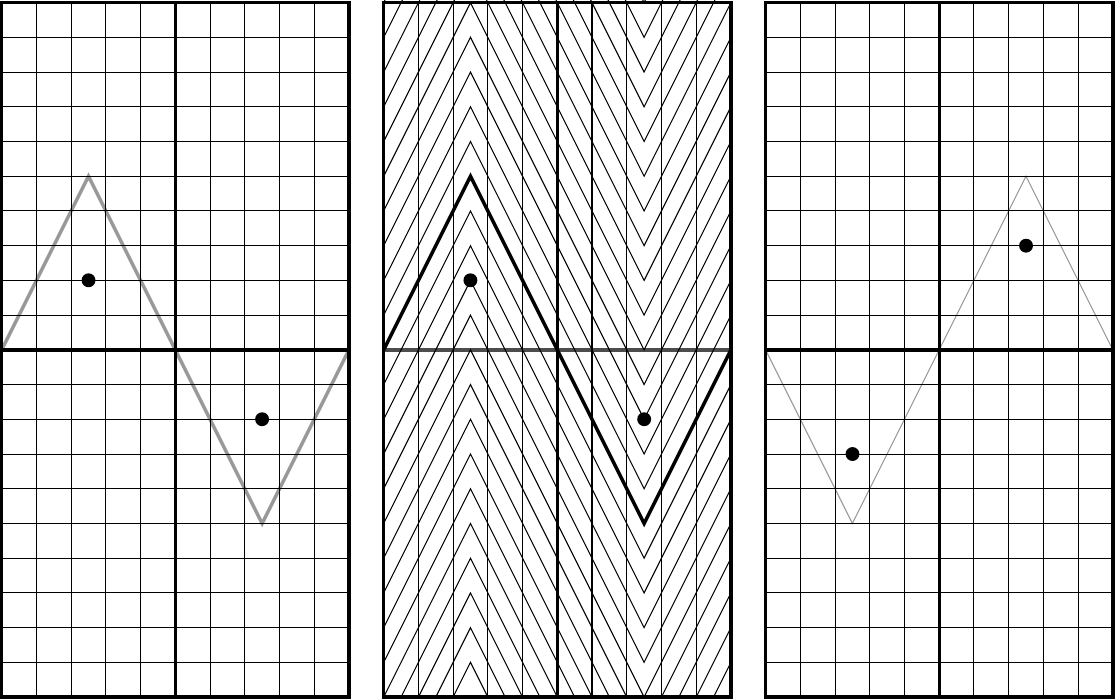}
\caption{The Ecalle height of the critical value can be changed by 
putting a different complex structure onto the Ecalle cylinder, and 
then by pull-backs onto the entire parabolic basin. Left: the 
critical orbit (marked by heavy dots) in the Ecalle cylinder, with a 
square grid indicating the complex structure; the equator is 
highlighted, and the critical value has Ecalle height $0.2$. The grey 
zig-zag line will be the new equator; it is invariant under $z\mapsto 
\zbar+1/2$.
Center: a grid of ``distorted squares'' defines a new complex 
structure (in which each parallelogram should become a rectangle); 
the dynamics is the same as before, and the new equator is 
highlighted. Right: the new complex structure in the Ecalle cylinder 
after straightening; the Ecalle height of the critical value is now 
$-0.3$. The image of the old equator is indicated in grey. }
\label{Fig:ChangeEcalleHeight}
\end{figure}

\section{Bifurcation Along Arcs and the Fixed Point Index}
\label{Sec:FixedPointIndex}


\begin{lemma}[Parabolic Arcs on Boundary of Odd Period Components]
\label{Lem:ParabolicArcLimits}
Near both ends, every limit point of every parabolic arc is a parabolic cusp.
\end{lemma}
\begin{proof}
Each limit point of parabolic parameters of period $k$ must be 
parabolic of period $k$, so it could be a simple or double parabolic. 
But at simple parabolics, Ecalle height is finite, while it tends to 
$\infty$ at the ends of parabolic arcs. Therefore, each limit point 
of a parabolic arc is a parabolic cusp.
\end{proof}

\begin{remark}
In fact, the number of parabolic cusps of any given (odd) period is 
finite \cite[Lemma~2.10]{Multicorns2}, so each parabolic arc has two 
well-defined endpoints.

As the parameter tends to the end of a parabolic arc, the Ecalle 
height tends to $\pm\infty$, and the Ecalle cylinders (with first 
return map of period $k$ which permutes the two ends) becomes 
pinched; in the limit, the cylinder breaks up into two cylinders that 
are interchanged by the $k$-th iterate, so each cylinder has a return 
map of period $2k$, which is holomorphic: the double parabolic 
dynamics in the limit is rigid and has no non-trivial deformations.
\end{remark}


In the sequel, we will need the \emph{holomorphic fixed point index}: 
if $f$ is a local holomorphic map with a fixed point $z_0$, then the 
index  $\iota(z_0)$ is defined as the residue of $\frac{1}{z-f(z)}$ at 
$z_0$. If the multiplier $\rho=f'(z_0)$ is different from $1$, this index 
equals $\frac{1}{1-\rho}$ and tends to $\infty$ as $\rho\to 1$. The 
most interesting situation occurs if several simple fixed points 
merge into one parabolic point. Each of their indices tends to 
$\infty$, but the sum of the indices tends to the index of the 
resulting parabolic fixed point, which is finite. Of course, 
analogous properties apply for the first return map of a periodic 
point.

If $z_0$ is a parabolic fixed point with multiplier $1$, then in local 
holomorphic coordinates the map can be written as 
$f(w)=w+w^{q+1}+\alpha w^{2q+1}+\dots$, and $\alpha$ is a conformal 
invariant (in fact, it is the unique formal invariant other than $q$: there is a formal, not necessarily convergent, power series that formally conjugates $f$ to its first three terms). A simple calculation shows that $\alpha$ equals the 
parabolic fixed point index. The quantity $1-\alpha$ is known as 
``r\'esidu it\'eratif'' \cite{BE}; its real part measures whether or 
not the parabolic fixed point of $f$ in the given normal form can be 
perturbed into $q$ or $q+1$ attracting fixed points; Epstein 
introduced the notion ``parabolic repelling'' and ``parabolic 
attracting'' for these two situations, and in the latter case he 
obtains an extra count in his refined Fatou-Shishikura-inequality 
\cite{AdamFSI}. We will use these ideas in 
Theorems~\ref{Thm:BifurcationArcs} and 
\ref{Thm:EcalleHeightZeroNoBifurcation} below.


\begin{lemma}[Types of Perturbation of Odd Period Parabolic Orbit]
\label{Lem:TypesPerturbation}
Suppose $p_{c_0}$ has a simple parabolic periodic point $z_0$ of odd period $k$. Then for any sequence $c_n\to c_0$ of parameters with $c_n\neq c_0$, the maps $p_{c_n}$ have periodic points $z_n$ and $z'_n$ that both converge to $z_0$ as $c_n\to c_0$ and with multipliers $\rho_n:=(p_{c_n}^{\circ 2k})'(z_n)\to 1$ and $\rho'_n:=(p_{c_n}^{\circ 2k})'(z'_n)\to 1$, such that for large $n$ either

\begin{itemize}
\item
both $z_n$ and $z'_n$ have period $k$, we have $\rho_n,\rho'_n\in\R$, and one of the orbits is attracting, while the other one is repelling; or
\item
the points $z_n=z'_n$ are on a parabolic orbit of period $k$; or
\item
the points $z_n$ and $z'_n$ both have period $2k$, they are on the same orbit of $p_{c_n}$, and they satisfy $p_{c_n}^{\circ k}(z_n)=z'_n$ and $p_{c_n}^{\circ k}(z'_n)=z_n$. Their multipliers satisfy $\rho'_n=\ovl{\rho_n}\not\in\R$ and $\Re(\rho_n-1)=O(\Im(\rho_n)^2)$.
\end{itemize} 
\end{lemma}
\begin{proof}
The point $z_0$ is a simple parabolic fixed point of the holomorphic map $p_{c_0}^{\circ 2k}$, so under small perturbations it must split up into exactly two fixed points of $p_{c_n}^{\circ 2k}$ (unless $c_n$ is some other parameter on the parabolic arc, where $z_n=z'_n$ are still parabolic). As periodic points of $p_{c_n}$, these must both have period $k$ or both period $2k$. They converge to the parabolic orbit, so their multipliers must tend to $1$ and their fixed point indices $1/(1-\rho_n)$ and $1/(1-\rho'_n)$ must tend to $\infty$. However, the sum of these indices must tend to the finite fixed point index of the parabolic periodic point.

If the period equals $k$, then the orbit of $p_c^{\circ 2k}(z_n)=z_n$ visits each of the $k$ periodic points twice: once for an even (holomorphic) and once for an odd (antiholomorphic) iterate, and the chain rule implies that the multiplier $\rho_n$ is real. The same argument applies to $z'_n$ and $\rho'_n$. The two fixed point indices are real and have large absolute values (once $\rho$ and $\rho'$ are close to $1$), so their sum can be bounded only if one index is positive and the other one negative; hence one orbit must be attracting and the other one repelling.

If the period equals $2k$, then the periodic points $z_n$ and $z'_n$ that are near $z_0$ must be on the same orbit with $p_{c_n}^{\circ k}(z_n)=z'_n$ and $p_{c_n}^{\circ k}(z'_n)=z_n$. A similar argument as above shows $\rho'_n=\ovl{\rho_n}$. For the sum of the fixed point indices, we obtain
\begin{equation}
\frac{1}{1-\rho_n}+\frac{1}{1-\ovl{\rho_n}}
=
{2\Re\left(\frac{1}{1-\rho_n}\right) =}
\frac{2(1-\Re\rho_n)}{|1-\rho_n|^2} \;.
\label{Eq:FixedPointIndex}
\end{equation}
Since his quantity must have a finite limit, the multipliers cannot be real. 
Writing $\eps_n:=1-\rho_n$, we have $\Re\eps_n=O(\eps_n^2)$, hence 
$\Re\eps_n=O(\Im\eps_n^2)$.
\end{proof}

In the following two results, we will show that both possibilities actually occur in every neighborhood of every simple parabolic parameter of odd period.


\begin{lemma}[Parabolics on Boundary of Hyperbolic Components]
\label{Lem:ParabolicsOnBoundary}
If a map $p_c$ has a parabolic periodic point of period $k$, then $c$ 
is on the boundary of a hyperbolic component of period $k$.
\end{lemma}
\begin{proof}
We will employ a classical argument by Douady and Hubbard.
Consider a map $p_{c_0}$ with a parabolic orbit of odd period $k$. To 
see that it is on the boundary of a hyperbolic component $W$ of 
period $k$, restrict the antipolynomial to an antipolynomial-like map 
of equal degree and perturb it slightly so as to make the indifferent 
orbit attracting: this can be achieved by adding a small complex 
multiple of an antipolynomial that vanishes on the periodic cycle but 
the derivative of which does not. Then apply the straightening 
theorem (Theorem~\ref{Thm:Straightening}) to bring it back into our family of maps $p_c$. This can be 
done with arbitrarily small Beltrami coefficients, so $c$ is near 
$c_0$ (see also \cite[Theorem~2.2]{Multicorns2}).
\end{proof}

Of course, the indifferent orbit can also be made repelling by the 
same reasoning. However, one would expect that a perturbation of a simple parabolic periodic point, here of period $k$, creates two periodic points of period $k$. In our case, this is possible whenever the perturbation goes into the hyperbolic component $W$, and then one of the two orbits after perturbation is attracting and the other is repelling; a perturbation creating two repelling period $k$ orbits is not possible within our family.
It turns out, though, that if $k$ is odd, then one can also perturb so that no nearby orbit of period $k$ remains --- and an orbit of period $2k$ is created. (The number of periodic orbits of given period must remain constant for perturbations of holomorphic maps such as $p_c^{\circ 2k}$, not for antiholomorphic maps such as $p_c^{\circ k}$.)


\begin{lemma}[Orbit Period Doubles in Bifurcation Along Arc]
\label{Lem:PeriodDoubling}
Every parabolic arc with a parabolic orbit of period $k$ (necessarily
odd) is the locus of transition where two periodic orbits of period
$k$ (one attracting and one repelling near the arc) turn into one
orbit of period $2k$ (attracting, repelling, or indifferent near the
arc). Equivalently, every parameter $c$ with a simple parabolic 
periodic orbit of odd period $k$ is on the boundary of a hyperbolic 
component $W$ of period $k$, and $c$ has a neighborhood $U$ so that 
for $c'\in W\cap U$, the parabolic orbit splits up into two orbits 
for $p_{c'}$ of period $k$ (one attracting and one repelling), while 
it splits into one orbit of period $2k$ for $c'\in U\sm\ovl W$.
\end{lemma}
\begin{proof}
As in Lemma~\ref{Lem:ParabolicsOnBoundary}, consider a map $p_{c_0}$ 
with a parabolic orbit of odd period $k$ and restrict it to an 
antipolynomial-like map of equal degree. This time, we want to perturb it so as to make the period $k$ orbit vanish altogether; therefore, we cannot just add a polynomial that takes the value zero along this orbit.


Let $z_0$ be one of the parabolic periodic points and change 
coordinates by translation so that $z_0=0$. By rescaling, we may 
assume that near $0$, we have $p_{c_0}^{\circ 
k}(z)=\zbar+A \zbar^2+o(\zbar^2)$ (note that conjugation by complex scaling changes the coefficient in front of $\zbar$; see the remark below). The second iterate has the local form $p_{c_0}^{2k}(z)= z+(A+\ovl A) z^2+o(z^2)$, so the assumption that the parabolic orbit is simple means $\Re A\neq 0$; conjugating if necessary by $z\mapsto -z$ we may assume that $\Re A>0$.

Let $z_{k-1}$ be the periodic point with $p_{c_0}(z_{k-1})=z_0$.
Let $f$ be an antipolynomial (presumably of large degree) that 
vanishes at the indifferent orbit except at $z_{k-1}$, where it takes the value $f(z_{k-1})=1$; assume further that the first and second derivatives of $f$ vanish 
at the entire indifferent orbit, and that 
$f$ vanishes to order $d$ at the critical point. For sufficiently small $\eps\in\C$, we will consider 
$f_\eps(z)=p_{c_0}(z)+\eps f(z)$; by slightly adjusting the domain 
boundaries, this will give an antipolynomial-like map of the same 
degree $d$ as before, and it will continue to have a single critical 
point of maximal order. The map $p_{c_0}$ has finitely many orbits of 
period $k$, all but one of which are repelling, and for sufficiently 
small $\eps$ these will remain repelling.

\hide{
\marginpar{Isn't there an easier argument. The map $p_c^{\circ 2k}$ looks like\break $z\mapsto z+z^2+\epsilon$
 near $0$, and as such has two fixed points near $0$. --- This is as easy as I could make it.} 
}
\marginpar{There was a nice parabolic index argument that no longer seems necessary?}

We claim there is a neighborhood $U$ of $z_0=0$ so that for sufficiently small $\eps>0$, the map $f_\eps$ will not have a point of period $k$ in this neighborhood. We use our local coordinates where $p_{c_0}^{\circ k}(z)=\zbar+A \zbar^2+o(\zbar^2)$. We may assume that $|p_{c_0}^{\circ k}(z)-\zbar -A\zbar^2|\le |Az^2|$ in $U$, and also $|y|<1/|8A|$. 

We have
\hide{
\[f_\eps^{\circ k}(z) = f_0^{\circ k}(0)+\eps (\partial/\partial\eps) f_\eps^{\circ k}(0)+\zbar(\partial/\partial \zbar)p_{c_0}^{\circ k}(0) = 0+ \eps^p\mu + \zbar +hot
\]
}
\[f_\eps^{\circ k}(z) 
\hide{= p_{c_0}^{\circ k}(z)+\eps (\partial/\partial\eps) f_\eps^{\circ k}(0) +hot} 
= p_{c_0}^{\circ k}(z)+ \eps +O(|z|^3)
\]
because $f$ vanishes to second order along the parabolic orbit. By restricting $U$ and $\eps$, we may assume that $|f_\eps^{\circ k}(z)-p_{c_0}^{\circ k}(z)-\eps| < (\Re A)|z|^2/2$. 

Now suppose $z=x+iy\in U$. If $|y|> 8|A|x^2$, then 
\[
|p_{c_0}^{\circ k}(z)-\zbar|\le 2|Az^2|=2|A|(x^2+y^2)<2|A|\left(\frac{|y|}{8|A|}+\frac{|y|}{8|A|}\right)=|y|/2
\] 
and 
\[
|f_\eps^{\circ k}(z)-\zbar-\eps| \le |f_\eps^{\circ k}(z)- p_{c_0}^{\circ k}(z)-\eps|+|p_{c_0}^{\circ k}(z)-\zbar | \le (\Re A)|z|^2/2+ 2|Az^2| \le |y|
\;,
\] 
so on this domain $p_{c_0}^{\circ k}$ and $f_\eps^{\circ k}$ behave essentially like complex conjugation (plus an added real constant) and thus have no fixed points.

However, if $|y|\le 8|A|x^2$, i.e., $z$ is near the real axis, then $\Re(A\zbar^2)\ge (\Re A) |z|^2/2$ and we have
\begin{align*}
\Re f_\eps^{\circ k}(z) &\ge \Re p_{c_0}^{\circ k}(z) +\eps - |f_\eps^{\circ k}(z)-p_{c_0}^{\circ k}(z)-\eps|
\ge \Re z + \Re(A\zbar^2) +\eps  +O(|z|^3)
\\
& \ge \Re z+\eps>\Re z.
\end{align*}
so that $f_\eps^{\circ k}$ does not have a fixed point with $|y|\le 8|A|x^2$ either. The size of the neighborhood $U$ is uniform for all sufficiently small $\eps$.

Now apply the straightening theorem as in the 
Lemma~\ref{Lem:ParabolicsOnBoundary}: this yields antipolynomials $p_{c_n}$ 
near $c_0$ with $c_n\to c_0$ for which there is one orbit of period 
$k$ fewer than before perturbation, and these are all repelling, so we are outside of $W$.
Since $z_0$ has period $k$ for the \emph{holomorphic} map 
$p_{c_0}^{\circ 2}$, there is a sequence $(z_n)$ of periodic points 
of period $k$ for $p^{\circ 2}_{c_n}$ with $z_n\to z_0$. These points must have period $2k$.

This shows that arbitrarily close to $c_0$ there are parameters for which the indifferent period $k$ orbit has turned into an orbit of period $2k$; similarly, by Lemma~\ref{Lem:ParabolicsOnBoundary} there are parameters near $c_0$ for which there is an attracting orbit of period $k$. Finally, by Lemma~\ref{Lem:TypesPerturbation}, any perturbation of $p_{c_0}$ away from the parabolic arc either creates an attracting orbit of period $k$ or an orbit of period $2k$ (which may be attracting, repelling, or indifferent; see Corollary~\ref{Cor:BifArcs} and Theorem~\ref{Thm:EcalleHeightZeroNoBifurcation} below). The transition between these two possibilities (attracting orbit of period $k$ vs.\ orbit of period $2k$) can happen only when the period $k$ orbit is indifferent, hence on the boundary of a hyperbolic component of period $k$. This proves the claim.
\end{proof}

\begin{remark}
The local behavior of an antiholomorphic map at a fixed point (or a periodic point of odd period) is quite different from the holomorphic case. If the fixed point is at $0$, such a map can be written $f(z)=a_1\zbar + a_2\zbar^2+\dots$; we will suppose $a_1\neq 0$. The second iterate takes the form $f^{\circ 2}(z)=|a_1|^2 z+\dots$, so indifferent orbits are always parabolic and $|a_1|$ is an invariant under conformal conjugations. However, $a_1$ itself is not: conjugating $u=\lambda z$ leads to 
\[
f_1(u)=(\ovl\lambda/\lambda) a_1\ovl u + (\ovl\lambda^2/\lambda) a_2\ovl u^2+\dots \;,
\]
so $\arg a_1$ depends on the rotation of the coordinate system (the linear approximation $df$ has eigenvalues $|a_1|$ and $-|a_1|$ with orthogonal eigenlines, and of course their orientation depends on the rotation of the coordinate system).

Specifically if $|a_1|=1$, we can choose $\lambda$ so that $f_1(u)=\ovl u+ A\ovl u^2+\dots$ with $A\in\C$, and conjugation by scaling can change $|A|$. Note that we have $f_1^{\circ 2}(u)=u+(A+\ovl A)u^2+\dots$.
If $\Re A\neq 0$, then there is a local \emph{quadratic} conjugation $v=au+bu^2$ with $a\in\R$ that brings our map into the form
$f_2(v)=\ovl v+\ovl v^2+\dots$, as can be checked easily. However, if $\Re A\neq 0$, there is no such change of coordinates because $f_1^{\circ 2}(u)=u+O(u^3)$, so the origin has a multiple parabolic fixed point.
\end{remark}

\hide{
\begin{corollary}[Parabolics on Boundary of Hyperbolic Components]
\label{Cor:ParabolicsOnBoundary}
If a map $p_c$ has a parabolic periodic point of period $k$, then $c$ 
is on the boundary of a hyperbolic component of period $k$.
\end{corollary}
\begin{proof}
This has been shown at the beginning of the previous lemma.
\end{proof}
}


\begin{proposition}[Fixed Point Index on Parabolic Arc]
Along any parabolic arc of odd period, the fixed point index is a 
real valued real-analytic function that tends to $+\infty$ at both 
ends.
\end{proposition}
\begin{proof}
The fact that the fixed point index is real valued follows for 
instance from (\ref{Eq:FixedPointIndex}). The Ecalle height 
parametrizes the arc real-analytically 
(Proposition~\ref{Prop:ParabolicArc}), and as the Ecalle height is 
changed by a quasiconformal deformation, the residue integral that 
defines the fixed point index depends analytically on the height (the 
integrand as well as the integration path). Therefore, the index 
depends real-analytically on Ecalle height.

The parabolic periodic point is simple for all parameters along the 
arc, but towards the end of a cusp the parabolic orbit merges with 
another repelling periodic point so as to form a double parabolic 
(Lemma~\ref{Lem:ParabolicArcLimits}). The orbit with which it merges 
is repelling, say with multiplier $\rho$, so its index $1/(1-\rho)$ 
tends to $\infty$ in $\C$. In order for the limiting double parabolic 
to have finite index, the index $\iota(z_n)$ of the parabolic orbit 
of $p_{c_n}$ must tend to $\infty$ as well as $c_n$ tends to the end 
of a parabolic arc.

Note that the index $\iota(z_n)$ is real by 
(\ref{Eq:FixedPointIndex}), so it tends to $+\infty$ or to $-\infty$. 
Since $|\rho|>1$,  the index $1/(1-\rho)$ always has real part less 
than $+1/2$. This implies that $\iota(z_n)\to+\infty$ (or the sum in 
the limit would not be finite).
\end{proof}


\begin{theorem}[Odd-Even Bifurcation and Fixed Point Index]
\label{Thm:BifurcationArcs}
Every parabolic arc of period $k$ intersects the boundary of a hyperbolic component of period $2k$ at the set of points where the fixed point index is at 
least $1$, except possibly at (necessarily isolated) points where the 
index has an isolated local maximum with value $1$.
\end{theorem}
\begin{proof}

Consider a parameter $p_{c_0}$ on a parabolic arc, and a sequence 
$c_n\to c_0$ so that all $p_{c_n}$ have all periodic orbits of period 
$k$ repelling. As in Lemma~\ref{Lem:PeriodDoubling}, let $z_0$ be a 
parabolic periodic point for $p_{c_0}$, let $z_n$ be a periodic point 
of period $2k$ for $p_{c_n}$ with $z_n\to z_0$, and let 
$z'_n:=p_c^{\circ k}(z_n)$. Let $\rho_n$ and $\rho'_n=\ovl{\rho_n}$ 
be the multipliers of $z_n$ and $z'_n$.
The sum of the two fixed point indices equals 
$2\Re\left(\frac{1}{1-\rho_n}\right)$. We have $|\rho_n|<1$ if and 
only if $2\Re\left(\frac{1}{1-\rho_n}\right)>1$.

Therefore, if $c_0$ is on the boundary of a period $2k$ component, we 
can choose $c_n$ so that $|\rho_n|<1$ and the fixed point index at 
$c_0$ is at least $1$. Conversely, if the fixed point index is 
greater than $1$, then we must have $|\rho_n|<1$ for all large $n$, 
and the limit is on the boundary of a period $2k$ component. If the 
index equals $1$, by real-analyticity of the index, either the index 
has an isolated local maximum there, or the point is a limit point of 
points with index greater than $1$ (note that the set of points with 
index $1$ is isolated as the index is real-analytic and  tends to 
$\infty$ at the ends).
\end{proof}

\begin{corollary}[Bifurcation Along Arcs]
\label{Cor:BifArcs}
Every parabolic arc has, at both ends, an interval of positive length 
at which a bifurcation from a hyperbolic component of odd period $k$ 
to a hyperbolic component of period $2k$ occurs.
\qed
\end{corollary}

\begin{corollary}[Boundary of Bifurcating Component Lands]
Let $W$ be a hyperbolic component of period $2k$ bifurcating from a 
hyperbolic component $W_0$ of odd period $k$, the set $\partial W\sm 
\ovl{W_0}$ accumulates only at isolated points in $\partial W_0$.
\qed
\end{corollary}
This rules out the possibility that the boundary curve of $W$ 
accumulates at $\partial W_0$ like a topologist's sine curve. The 
reason is that the set of limit points must have fixed point index 
$1$, and the set of such points is discrete.
\marginpar{\vskip -2 cm comment that there should be just two such points? A picture?}

\section{Parabolic Perturbations}
\label{Sec:ParabolicPerturbations}

In this section, we fix a hyperbolic component $W$ of odd period $k$ 
and a parameter $c_0\in\partial W$ with a simple parabolic orbit 
$z_0, \dots,z_{k-1}$. It is well known that there exists a 
neighborhood $V$ of $z_0$ and a local coordinate $\phi_{c_0}\colon 
V\to \C$ such that 
\[
\phi_{c_0}\circ p_{c_0}^{\circ 2k}\circ 
\phi_{c_0}^{-1}(\zeta)=\zeta+\zeta^2h(\zeta)
\] 
with $\phi_{c_0}(z_0)=0$, 
$h(0)=1$ and $|h(\zeta)-1|<\eps$ on $V$. We need to establish similar 
local coordinates for the local dynamics after perturbation.

\begin{proposition}[Perturbed Parabolic Dynamics]
\label{Prop:PerturbedFatouCoords2}
For every $\eps>0$, one can choose neighborhoods $V$ of $c_0$ and $U$ of $z_0$  so that 
 there is a 
$\phi_c\colon V\to\C$ that satisfies
\[
f_c(\zeta):=\phi_c\circ p_c^{\circ 2k}\circ
\phi_c^{-1}(\zeta)=\zeta+(\zeta^2-a_c^2)h_c(\zeta)
\]
with
$|h_c(\zeta)-1|<\eps$ on $V$ and $a_c\in\C$.
\end{proposition}
\begin{proof}
Since $p_c^{\circ 2k}$ is holomorphic, after perturbation the 
parabolic fixed point splits up into two simple fixed points in the domain of $\phi_{c_0}$. 
Let $\phi_c:= \phi_{c_0}+b$ where $b$ is chosen so the images of  these fixed points are symmetric to $0$, i.e., at some $\pm  a_c\in\C$
(note that $a_c$ may not be a continuous function in a neighborhood of $c$, but $a_c^2$ is). Then 
\[
h_c(\zeta):=\frac{\left(\phi_c\circ p_c^{\circ 2k}\circ \phi_c^{-1}\right)(\zeta)-\zeta}{(\zeta-a_c)(\zeta+a_c)}
\]
must be  holomorphic on $V$, and the map $h_c$ is close to $h$ at least for $\zeta\neq 0$, hence also near $0$.
\end{proof}

Write $U^+:=U\cap W$ and $U^-:=U\sm \ovl W$ (the parts inside and 
outside of the hyperbolic component $W$). 
By Lemma~\ref{Lem:TypesPerturbation}, for parameters $c\in 
U^-$ the parabolic orbit splits up into one orbit of period $2k$; 
denote it by $w_0(c),\dots,w_{2k-1}(c)$. By restricting $U$, we may 
assume that $|w_0(c)-z_0(c_0)|<\eps$ and $|w_k(c)-z_0(c_0)|<\eps$. 
Moreover, the multipliers $\rho_c:=(p_c^{\circ{2k}})'(z_0)$ and 
$\rho'_c:=(p_c^{\circ{2k}})'(z_k)$ are complex conjugate and 
$|\Re(\rho_c-1)|\in O(|\Im(\rho_c)|^2)$. Since $f'_c(a_c)=1+2a_c 
h_c(a_c)\in\{\rho_c,\rho'_c\}$ and $h_c$ is close to $1$, we see 
that $a_c$ is almost purely imaginary.

Let $L_c$ be the straight line through $a_c$ and $-a_c$ when $c\in{U^-}$; when $c\in\partial W$, let $L_c$ be the eigenline for eigenvalue $-1$ for the parabolic fixed point (every antiholomorphic fixed point with multiplier $1$ has eigenvalues $+1$ and $-1$). This family of lines is continuous in $c$ for $c\in \ovl{U^-}$ (a one-sided neighborhood of $c_0$), and $L_c$ is vertical  (in $\zeta$-coordinates) when $c\in \partial W$. 
For $c\in U^-$, let $\ell_c$ be the segment  $[a_c,-a_c]\subset L_c$. 

Choose $r>0$, and consider the arc of circle $K_c^{in}$ connecting $a_c$ to $-a_c$  going through $r$, and $K_c^{out}$ connecting the same two points through $-r$.   If $U$ is chosen so small that $a_c$ is almost purely imaginary for  $c\in U^-$, these arcs are well defined, and  as $a_c\to 0$, each of these  arcs has a limit, which is the circle through $0$ and centered at $\pm r/2$; see Figure~\ref{Fig:TwoCircles}.

\begin{figure}
\framebox{\Includegraphics{trim=0 0 0 0}{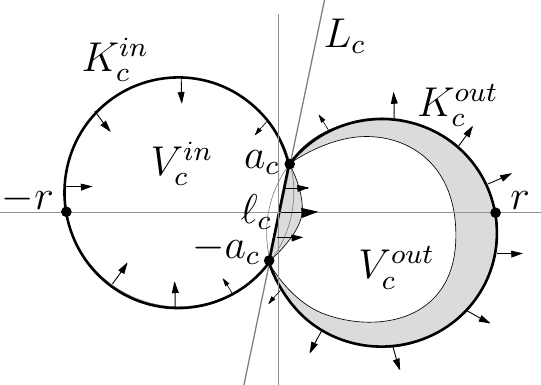}}
\caption{The line segment $\ell_c$ joining $a_c$ to $-a_c$ with its image under $f_c$; the region between these (shaded) is a fundamental domain for the dynamics. The arcs $K_c^{in}$ and $K_c^{out}$ are also shown, as well as an inverse image of $K_c^{out}$ and another fundamental domain bounded by $K_c^{out}$ and its inverse image.
This latter fundamental domain has a non-vanishing limit as $a_c\to 0$. }
\label{Fig:TwoCircles}
\end{figure}

Denote by $V_c^{in}$ the region bounded by $K_c^{in}$ and $\ell_c$, and  by $V_c^{out}$ the region bounded by $K_c^{out}$ and $\ell_c$.  There is a fixed choice of $r>0$ so that for all $c\in U^-$, iterates of $\zeta\in V_c^{in}$ under $f_c$ will remain in $V_c^{in}$ until they exit to $V_c^{out}$ through $\ell_c$, and similarly iterates of $\zeta\in V_c^{out}$ under $f_c^{-1}$ will remain in $V_c^{out}$ (for $c\in U\cap \partial W$ these iterates never exit at all). The details of this argument are somewhat tedious but not difficult: as soon as $|a_c|\ll r$, the iterates in the first quadrant of $K_c^{out}$ move upwards and to the right, while the iterates in the second quadrant (up to the point $a_c$) move upwards and to the left --- except on a short piece of arc near the top of the circle, where the iteration moves essentially upwards. The argument is similar for the other parts of $K_c^{out}$ and $K_c^{in}$.


\begin{proposition}[Ecalle Cylinders After Perturbation] \label{prop:pertFatoucoord} For every 
$c\in \ovl {U^-}$, the quotients $C^{in}_c:=V^{in}_c/p_c^{\circ 
2k}$ and $C^{out}_c:=V^{out}_c/p_c^{\circ 2k}$ 
(the quotients of $V_c^{in}$ and $V_c^{out}$ by the dynamics,  
identifying points that are on the same finite orbits entirely in 
$V^{in}_c$ or in $V^{out}_c$) are complex annuli isomorphic to 
$\C/\Z$.
\end{proposition}

\begin{proof}
In the 
parabolic case of $c\in \partial W$, this is a standard result, 
proved using Fatou coordinates (see Milnor~\cite[Sec.~10]{MiBook}). 
We will thus focus on the case $c\in U^-$.

Since the points $\pm a_c$ are almost purely imaginary and $h_c$ is 
almost $1$, it is easy to understand the dynamics of $f_c$, 
represented in Figure~\ref{Fig:TwoCircles}. In 
particular, the line segment $\ell=[-a_c,a_c]$ is sent by $f_c$ to the arc $f_c(\ell)\subset 
\ovl{V_c^{out}}$ still joining $a_c$ to $-a_c$ but disjoint from $\ell$ 
(except at the endpoints).  Let $A^{out}$ be the domain bounded by 
$\ell$ and $f_c(\ell)$. Identifying the two boundary edges of 
$A^{out}$ by $f_c$, we obtain a complex annulus that represents 
$C_c^{out}=A^{out}/p_c^{\circ 2k}$: every finite orbit in $V^{out}$ 
enters $A^{out}\cup f(\ell)$ exactly once.

Finally, we must see that $A^{out}/f_c$ is a bi-infinite annulus, 
i.e., isomorphic to $\C/\Z$ so that its ends are {punctures}.

This follows from the following lemma.

\begin{lemma} Let $\lambda$ be a non-real complex number, and let $g:u\mapsto \lambda u+O(u^2)$ be an analytic map defined in some neighborhood of $0$.  Let $\widetilde Q \subset \C\sm \{0\}$ be the region  bounded by $[0,r]$, $g([0,r])$ and $[r, g(r)]$, that we will take to include $(0,r)$ and $g((0,r))$ but not $(r,g(r))$. Then for $r$ sufficiently small, $(0,r]$ and $g((0,r])$ are disjoint, so that the quotient of $\widetilde Q$ by the equivalence relation identifying $x\in (0,r)$ to $g(x)$ is homeomorphic to an annulus, and and it has infinite modulus.
\end{lemma}
\begin{proof}
The proof consists of passing to $\log u$ coordinates, where the corresponding annulus is bounded by $(-\infty,\log r)\subset\R$ and the image curve $\log g(u)=\log u+\log\lambda+O(u)$; when setting $t=\log u$, the boundary identification relates $t\in\R^-$ to the image curve $t+\log\lambda+e^t$ (for $t\ll -1$). The claim follows.
\end{proof}

This lemma clearly applies to both ends of $A^{out}$, and the argument about $A^{in}$ is the same. This proves the Proposition.
\end{proof}

\newpage

Recall the definition of Ecalle cylinders in the parabolic attracting 
and repelling petals of the holomorphic map $p_c^{\circ 2k}$ 
(Lemma~\ref{LemEcalleCylinders}), and the statement that the 
antiholomorphic iterate $p_c^{\circ k}$ introduces an antiholomorphic 
self-map of the Ecalle cylinders that interchanges the two ends. The 
situation is similar here: since the map $p_c^{\circ k}$ commutes 
with $p_c^{\circ 2k}$, it induces antiholomorphic self-maps from 
$C^{in}_c$ to itself and from $C^{out}_c$ to itself. As $p_c^{\circ 
k}$ interchanges the two periodic points at the end of the cylinders, 
it interchanges the ends of the cylinders, so it must fix a geodesic 
in the cylinders $\C/\Z$ that we call again the \emph{equator}. 
Choosing complex coordinates in the cylinders for which the equator 
is at imaginary part $0$, we can again define \emph{Ecalle height} as 
the imaginary part in these coordinates. We will denote the Ecalle 
height of a point $z\in C^{in/out}_c$ by $E(z)$.

Since our arcs of circle $K_c^{in/out}$ depend continuously on $c$ and have a finite non-zero limit as $a_c$ tends to $0$, the construction of the perturbed Ecalle cylinders depends continuously on $c\in \ovl 
{U^-}$.  We summarize this in the following proposition. 

\begin{proposition}[Bundle of Ecalle Cylinders] The disjoint unions
\[
\mathcal C^{in}:=\bigsqcup_{c\in \ovl {U^-}}  C^{in}_c \quad \text 
{and}\quad  \mathcal C^{out}:=\bigsqcup_{c\in \ovl {U^-}}  C^{out}_c
\]
form $2$-dimensional complex manifolds with boundary, and the natural maps
\[
\mathcal C^{in}\to\ovl{U^-} \quad \text {and}\quad \mathcal C^{out}\to\ovl{U^-}
\]
are smooth morphisms that make $\mathcal C^{in}$ and $\mathcal C^{out}$ 
into topologically trivial bundles with fibers isomorphic to $\C/\Z$.

The equators form subbundles of circles, and the Ecalle height of fixed points in $\C$ near $0$ depends continuously on $c$.
\end{proposition}

\begin{remark}
Here ``smooth morphism'' means that $\mathcal C^{in}$ and $\mathcal C^{out}$ are families of complex manifolds parametrized by $\ovl{U^-}$ and that the fibers have analytic local coordinates that depend continuously on the parameter.
\end{remark}
 
Of central importance to us is that above $U^-$ (not the closure 
$\ovl{U^-}$)  the two bundles  $\mathcal C^{in}$ and $\mathcal 
C^{out}$ are canonically isomorphic as follows.

\begin{definition}[The Transit Map]
The \emph{transit map} is the conformal isomorphism
\[
T_c\colon C_c^{in}\to C_c^{out}
\]
induced by the conformal isomorphism $p_c^{\circ 2k}\colon A^{in} \to A^{out}$.
\end{definition}
This transit map clearly depends continuously on the parameter $c\in 
U^-$ and preserves the equators, hence Ecalle heights.


Finally, choose a smooth real curve $s\mapsto c(s)$ in $\ovl{U^-}$ (in parameter space), 
para\-me\-trized by $s\in [0,\delta]$ for some $\delta>0$, with $c(s)\in 
U^-$ for $s>0$.
Choose a smooth curve $s\mapsto \zeta(s)$ (in the dynamical planes, typically the critical value), also defined for $s\in 
[0,\delta]$ such that
\[
\zeta(s)\in V^{in}(c(s))
\]
  for all $s\in [0,\delta]$. Then $s\mapsto \zeta(s)$ induces a map 
$\sigma:[0,\delta]\to \mathcal C^{in}$ with $\sigma(s)\in 
C^{in}_{c(s)}$.

\begin{proposition}[Limit of Perturbed Fatou Coordinates]
\label{Prop:PerturbedFatouCoords}
The curve
\[
\gamma:= s\mapsto T_{c(s)}(\sigma(s))
\]
in $\mathcal C^{out}$, parametrized by $s\in (0,\delta]$, spirals as 
$s\searrow 0$ towards the circle on $C^{out}_{c_0}$ at Ecalle height 
$E(\sigma(0))$.
\end{proposition}

Before proving this we need to say exactly what ``spirals'' means. We know that $\mathcal C^{out}$ is a trivial topological bundle of bi-infinite annuli $\C/\Z$ over $\overline U^-$; we can choose a trivialization 
\[
\Phi:\mathcal C^{out} \to \ovl {U^-}\times \C/\Z
\]
by deciding that the point $r$ (see Figure~\ref{Fig:TwoCircles}) corresponds for all $c\in \ovl U^-$ to the origin of $\C/\Z$.  That allows us to define an {\it Ecalle phase\/} $\arg(\gamma(s))$ to be a continuous lift $\phi$ of
\[
s\mapsto \re (pr_2 (\Phi(\gamma(s)))) \in \R/\Z.
\]
Spiralling will mean that the image of $\gamma$ accumulates exactly on the circle  on $\mathcal C^{out}_{c_0}$ at Ecalle height $E(\sigma(0))$, and that in the process the Ecalle phase tends to infinity.

\begin{proof} Since the transit map preserves Ecalle heights, the 
curve $t\mapsto \gamma(s)$ can only accumulate on the circle on 
$C^{out}_{c_0}$ at Ecalle height $E(\sigma(0))$.  It remains to show
that the Ecalle phase tends to infinity. The magnitude of the Ecalle phase essentially measures how many iterates of $f_c$ it takes for $\zeta(s)$ to reach the fundamental domain in $V^{out}_c$ shown in Figure \ref{Fig:TwoCircles}. 

This more or less obviously tends to infinity as $a_c \to 0$;  to get from $V_c^{in}$ to $V_c^{out}$, the orbit must cross $\ell_c$, and near $\ell_c$ the map $f^{\circ k}_c$ moves points less and less as $a_c\to 0$. (In the language of Douady ``it takes longer and longer to go through the egg-beater''.)
In fact, in the dynamics of the limit $c_0$ it takes infinitely many iterations of $f_c^{-1}$ for $r$ to get to the origin, and thus arbitrarily many iterations to reach any small neighborhood $X$ of the origin, and for sufficiently small perturbations the number of backwards iterations to go from $r$ into $X$ is almost the same.
\end{proof}

\begin{remark}
A computation in logarithmic coordinates shows that the Ecalle phase $\arg\gamma(s))$ (essentially the number or iterations required to get from $\zeta(s)$ to the fundamental domain in $V_c^{out}$ containing $r$) satisfies  
\[
\arg\gamma(s) = \frac \pi{|a_c(s)|}(1+o(1)) 
\]
as $a_c\to 0$.
\end{remark}

\hide{
We can easily make this precise and estimate the Ecalle phase; compare Figure \ref{Fig:PerturbedLocalDynamics}. Consider the point $\zeta(s)$ in $V^{in}(c(s))$; we want to estimate the number of iterations until the orbit, after traversing the eggbeater, reaches the fixed fundamental domain in $V^{out}(c(s))$ around the point $u_0$ (shaded in dark grey in Figure~\ref{Fig:PerturbedLocalDynamics}).
}
\hide{
In $u$-coordinates, $\zeta(s)$ becomes approximately $1/\zeta(s)$ in the left half plane and moves in each iteration approximately by $-1$, until it reaches the line $\ell$ which in $u$-coordinates is at $\re u=-\pi/2a_c$, so this takes approximately $\pi/2|a_c|$ iterations (up to a bounded factor close to $1$). The line $\ell$ is represented in the $u$-plane twice, once each in the left and right half planes, corresponding to $V^{in}$ and $V^{out}$. As the orbit of $\zeta(s)$ traverses $\ell$, in the $u$-plane it enters the far right half plane (at real parts approximately $\pi/2|a_c|$) and again takes time about $\pi/2|a_c|$ until it reaches the fundamental domain of $u_0$. 
This  shows that the Ecalle phase $\phi(s)$ tends to infinity and that there exist constants $C_1,C_2>0$ with 
\[
C_1 \frac \pi{|a_c(s)|}\le  \phi(s) \le C_2 \frac \pi{|a_c(s)|} \;.
\]
In fact it isn't too hard to show that it has the asymptotic expansion
\[
\phi(s) = \frac \pi{|a_c(s)|}(1+o(1)) 
\]
as $a_c\to 0$.
}

\section{Parabolic Trees and Combinatorics}
\label{Sec:Combinatorics}

\begin{definition}[Characteristic Parabolic Point and Principal Parabolic]
The \emph{characteristic point} on a parabolic orbit is the unique 
parabolic periodic point on the boundary of the critical value Fatou 
component.

A map with a parabolic orbit is called a \emph{principal parabolic} 
if the parabolic orbit is simple and each point on the parabolic 
orbit is the landing point of at least two periodic dynamic rays.
\end{definition}

\begin{figure}[htbp]
\Includegraphics{height=40mm,trim=0 30 0 40,clip}{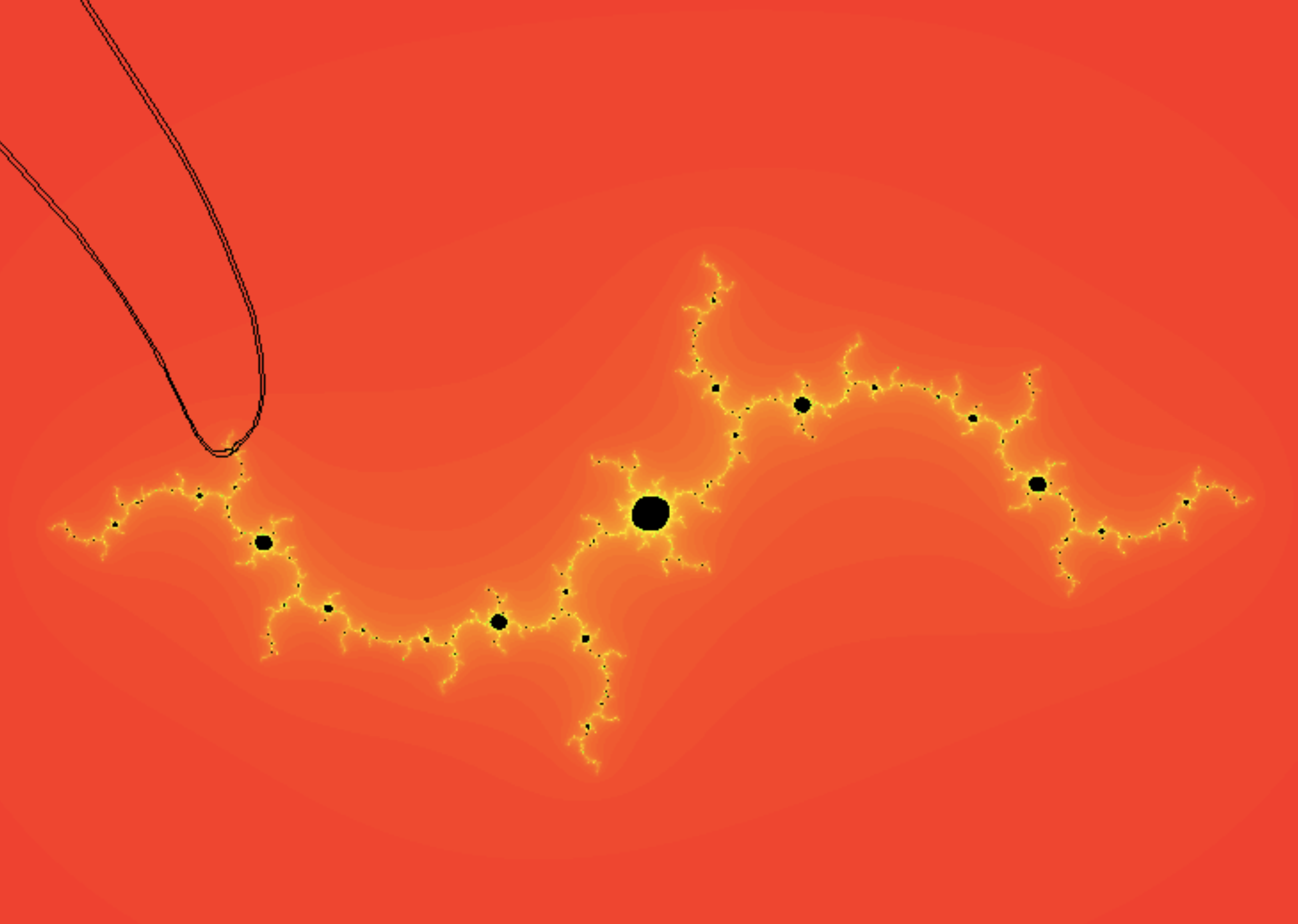}
\Includegraphics{height=40mm,trim=0 20 0 20,clip}{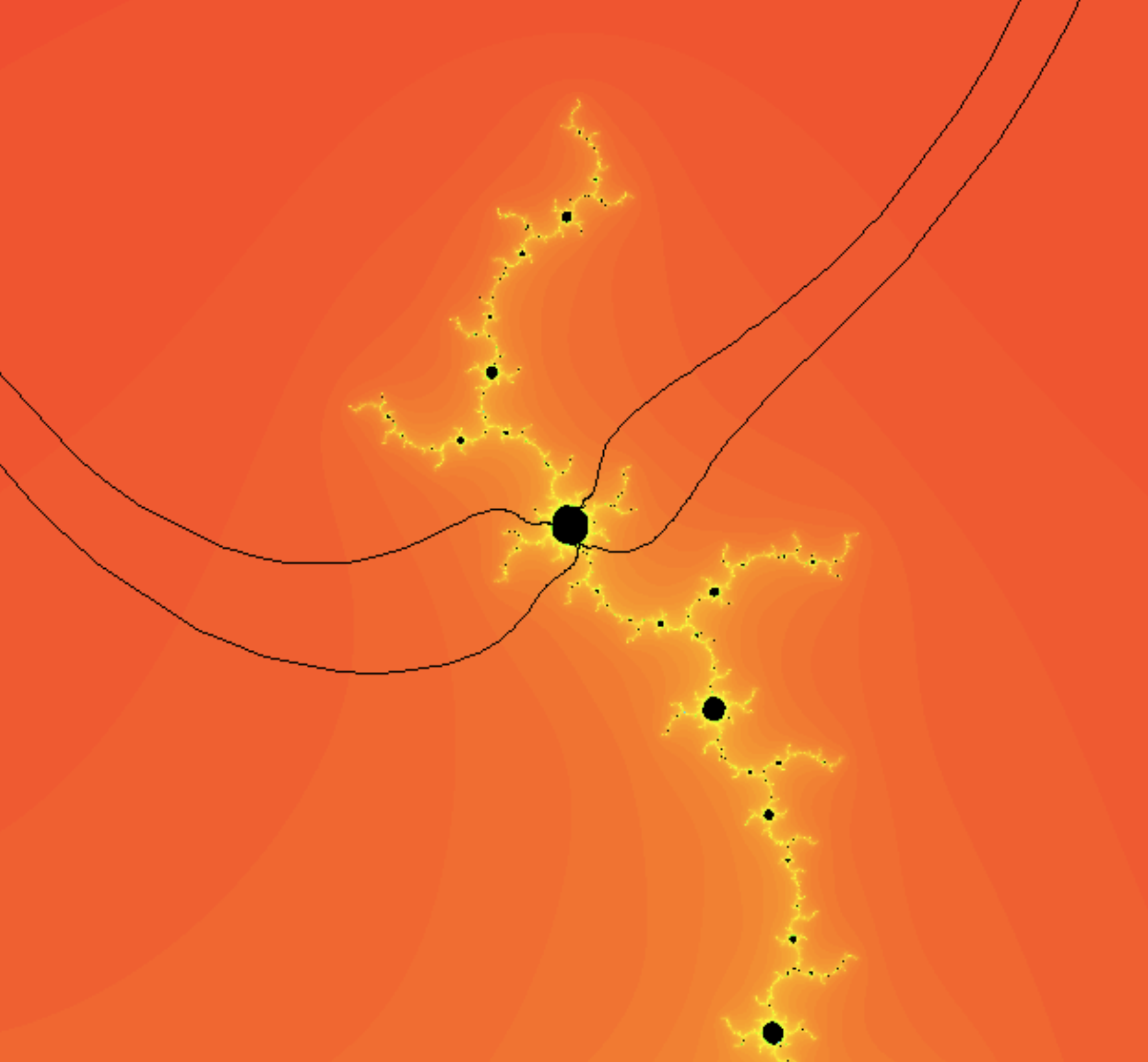}
\Includegraphics{height=50mm,trim=0 35 0 40,clip}{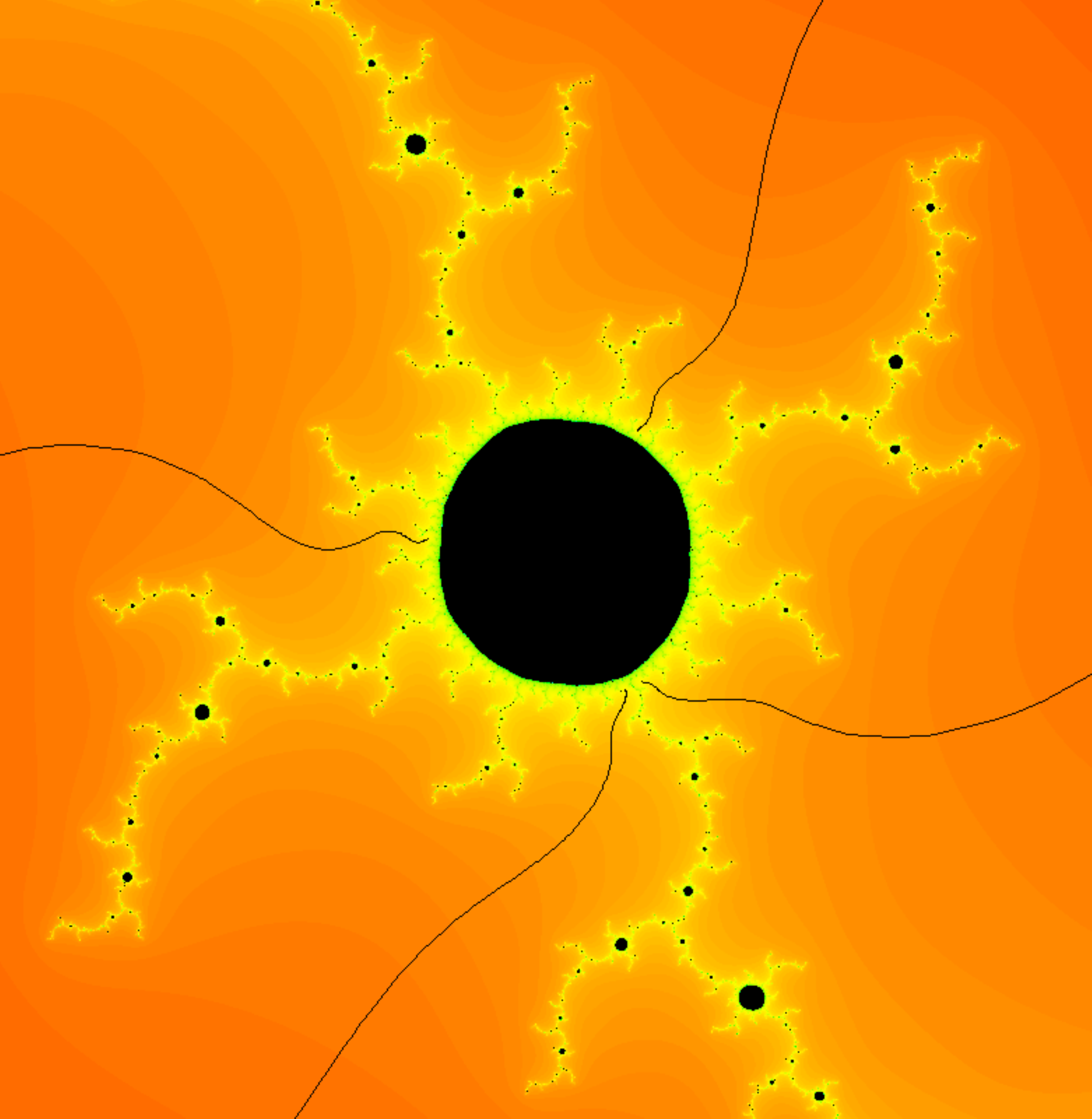}
\Includegraphics{height=50mm,trim=50 0 60 0,clip}{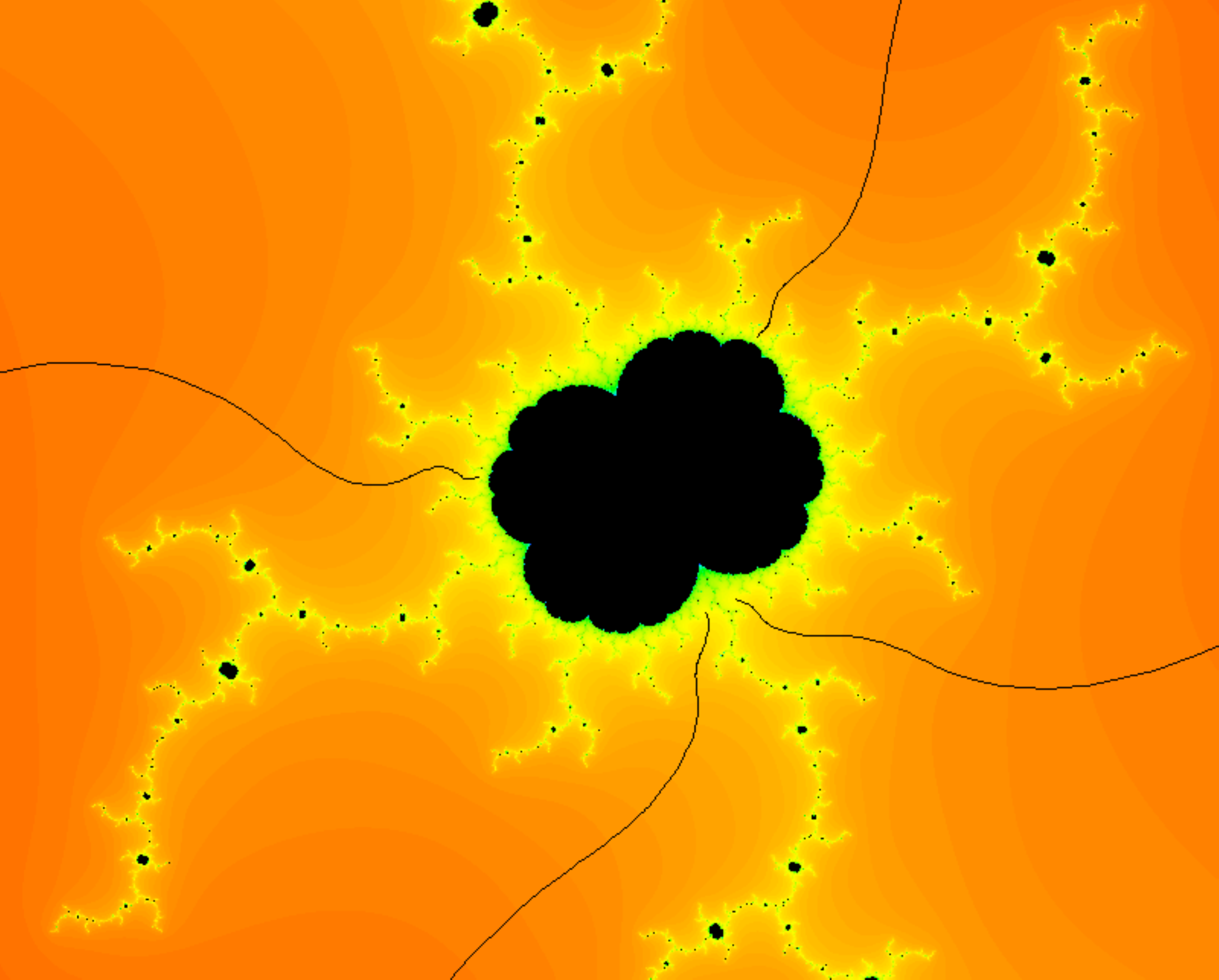}
\Includegraphics{height=50mm,trim=0 20 0 0,clip}{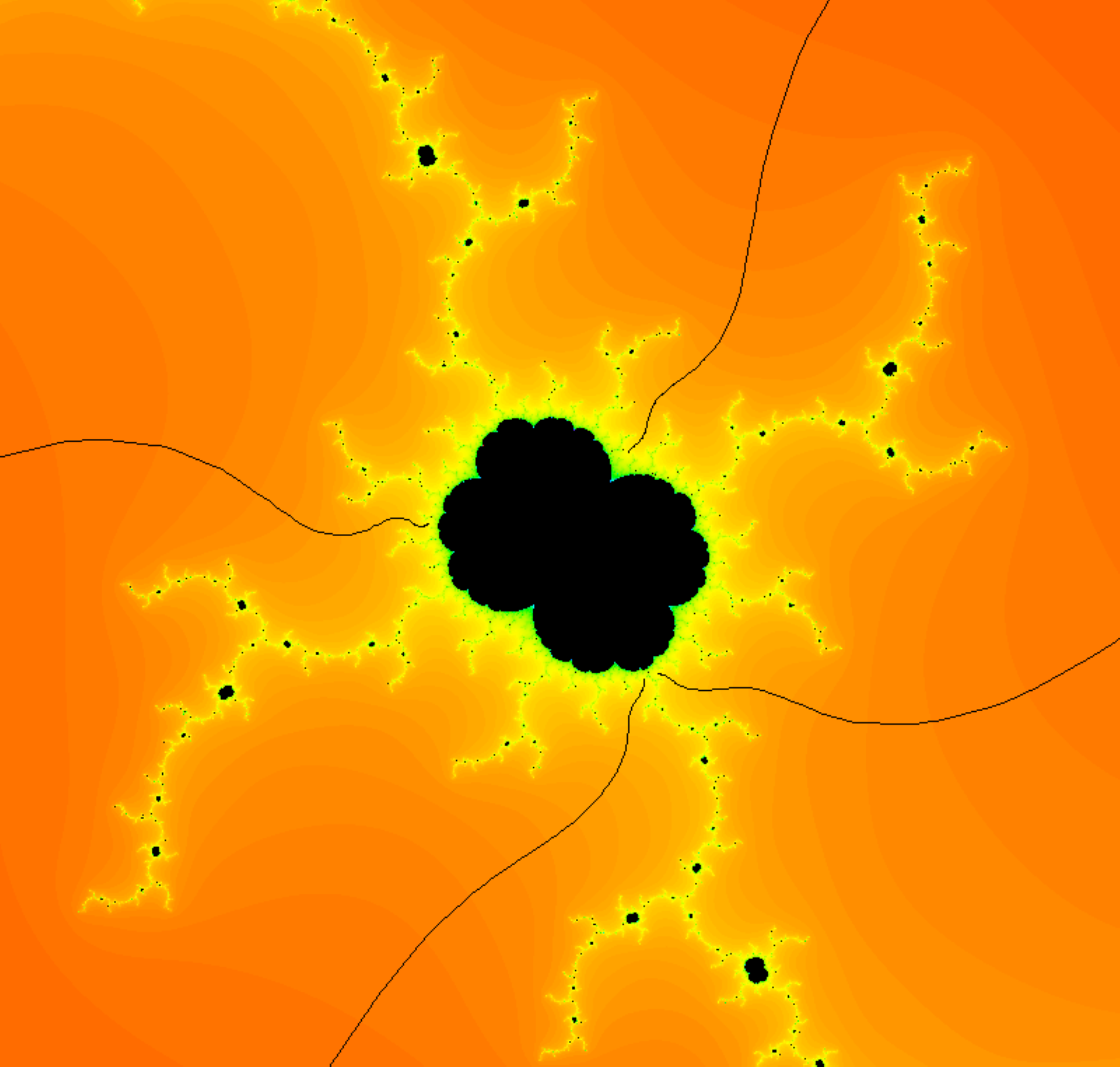}
\Includegraphics{height=50mm,trim=50 0 0 0,clip}{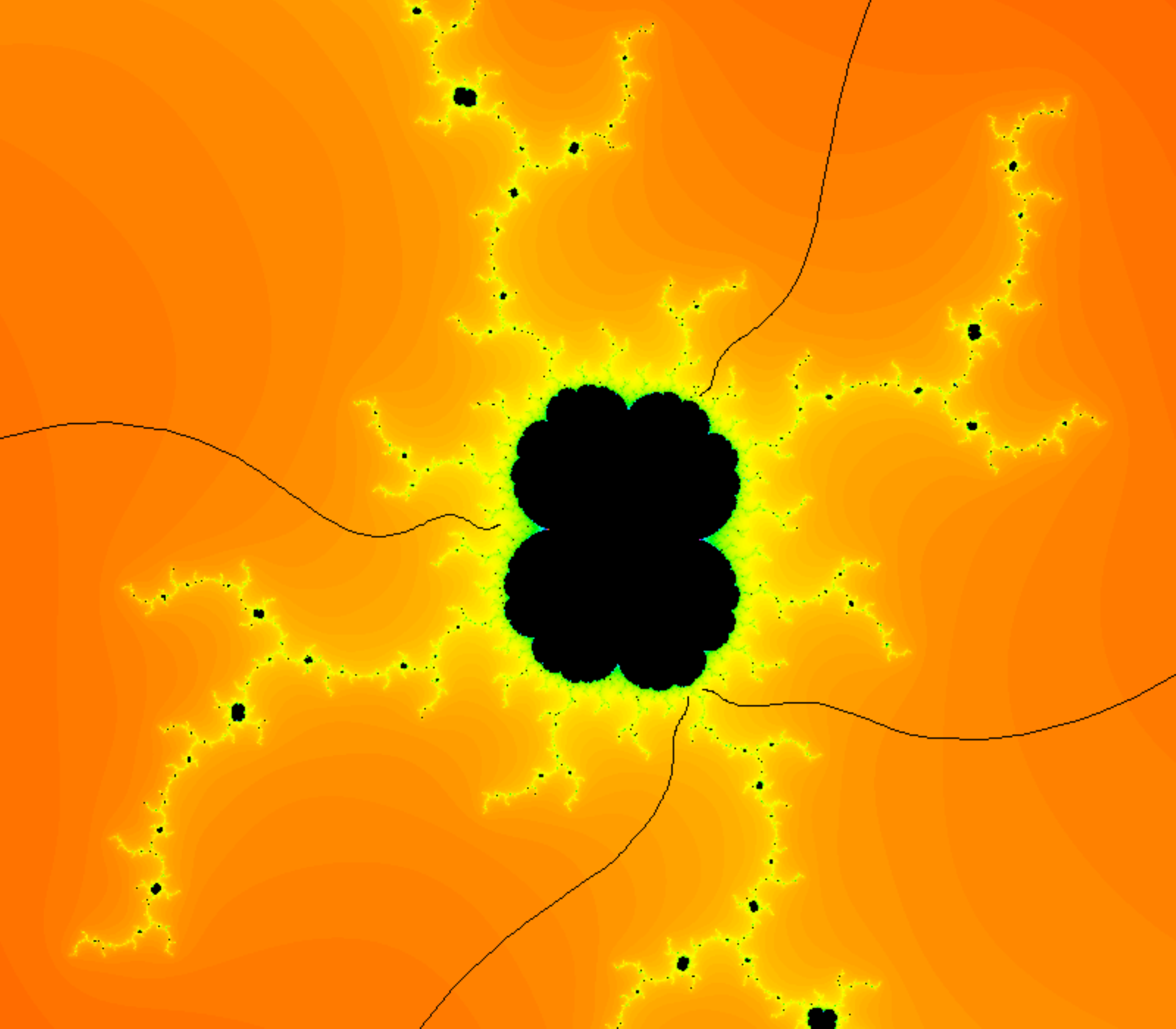}
\caption{An antiholomorphic map $p_c(z)=\zbar^2+c$ of degree $d=2$ 
with an attracting cycle of period $5$. The Fatou component 
containing the critical value has $d+1=3$ boundary points that are 
fixed under $p_c^{\circ 5}$, and together these are the landing 
points of $d+2=4$ dynamic rays: the dynamic root is the landing point 
of $2$ rays (here, at angles $371/1023$ and $404/1023$ of period 
$10$, and the two dynamic $2$-roots are the landing points of one ray 
each (at angles $12/33$ and $13/33$ of period $5$). Upper left: the entire 
Julia set with the four rays indicated. Upper right: blow-up of a 
neighborhood of the critical value Fatou component where the four 
rays can be distinguished. The hyperbolic component containing the 
parameter $c$ is bounded by $d+1$ parabolic arcs (see 
Figure~\ref{Fig:BigAndLittleTricorn}): one arc contains the 
accumulation set of the parameter rays at angles $12/33$ and $13/33$ 
(the root arc), and the other two arcs contain the accumulation sets 
of one parameter ray each (at angles $12/33$ and $13/33$ 
respectively). The four remaining pictures show further blow-ups near the critical value, for parameters at the center (middle row, left), from the parabolic root arc (middle right) and from the two parabolic co-root arcs (bottom row).}
\label{Fig:PrincipalNonPrincipalParabolics}
\end{figure}

\begin{remark}
As proved in \cite{Multicorns2}, every hyperbolic component $W$ of 
odd period $k$ in the Multicorn $\Md$ has a Jordan curve boundary consisting of 
exactly $d+1$ parabolic arcs and $d+1$ parabolic cusps where the arcs 
meet in pairs. Suppose $k\ge 3$. For each $c\in W$, each periodic 
bounded Fatou component has exactly $d+1$ boundary points that are 
fixed under the first return map of the component, and these points 
together are the landing points of $d+2$ periodic dynamic rays: one 
boundary fixed point is the landing point of two rays, both of period 
$2k$, and called the \emph{dynamic root} of the component, and the 
other boundary fixed points are the \emph{dynamic co-roots} and 
landing points of one ray each, of period $k$. Specifically for the 
critical value Fatou component, the two rays landing at the dynamic 
root separate this Fatou component and its co-roots from the entire 
critical orbit except the critical value (see Figure~\ref{Fig:PrincipalNonPrincipalParabolics}).
For each $c\in W$, the set of dynamic root and co-roots of the Fatou 
component containing the critical value is in natural bijection to 
the parabolic boundary arcs of $W$: at each of the $d+1$ boundary 
arcs of $W$, a different one of the dynamic roots or co-roots becomes 
parabolic. The \emph{parabolic root arc} is the arc at which the 
dynamic root becomes parabolic, while the $d$ \emph{co-root arcs} are 
those where one the $d$ co-roots becomes parabolic. Principal 
parabolic maps $p_c$ are thus maps from the root arc, and they exist 
on the boundary of each odd period component.
At a parabolic cusp, a dynamic root or co-root merges with one of its 
adjacent dynamic (co-)roots. Specifically at a cusp at the end of the 
root arc, the dynamic root merges with a co-root: at such parameters, 
each parabolic periodic point is the landing point of two rays of 
period $2k$ and one ray of period $k$.

\end{remark}

\begin{definition}[Parabolic Tree]
If $p_c$ has a principal parabolic orbit of odd period $k$, we define its 
\emph{parabolic tree} as the unique minimal tree within the filled-in 
Julia set that connects the parabolic orbit and the critical orbit, 
so that it intersects the critical value Fatou component along a 
simple $p_c^{\circ k}$-invariant curve connecting the critical value 
to the characteristic parabolic point, and it intersects any other 
Fatou component along a simple curve that is an iterated preimage of 
the curve in the critical value Fatou component. A \emph{loose 
parabolic tree} is a tree that is homotopic to the parabolic tree, by 
a homotopy that fixes the Julia set (so it acts separately on bounded 
Fatou components). It is easy to see that the parabolic tree 
intersects the Julia set in a Cantor set, and these points of 
intersection are the same for any loose tree (not that for simple parabolics, any two periodic Fatou components have disjoint closures).
\end{definition}

This tree is defined in analogy to the Hubbard tree for 
postcritically finite polynomials. In our case, note first that the 
filled Julia set is locally connected hence path connected, so any 
minimal tree connecting the parabolic orbit is uniquely defined up to 
homotopies within bounded Fatou components. The parabolic tree is 
$p_c$-invariant (this is clear by construction separately in the 
Julia set and in the Fatou set). A simple standard argument 
(analogous to the postcritically finite case) shows that the critical 
value Fatou component has exactly one boundary point on the tree (the 
characteristic parabolic point), and all other bounded Fatou 
components have at most $d$ such points (the preimages of the 
characteristic parabolic point). The critical value is an endpoint of 
the parabolic tree. All branch points of the parabolic tree are 
either in bounded Fatou components or repelling (pre)periodic points; 
in particular, no parabolic point (of odd period) is a branch point.

\begin{definition}[OPPPP: Odd Period Prime Principal Parabolic]
A principal parabolic map $p_c$ with a parabolic orbit of odd period 
$k\ge 3$ is called \emph{prime} if the parabolic tree does not have 
any proper connected subtree that connects at least two Fatou 
components and that is invariant under some iterate of $p_c$.
\end{definition}

Parabolics with these properties will be called OPPPP-parabolics, and 
these are the ones that we will work with.

\begin{remark}
The condition of ``prime'' can be motivated informally as follows.
Just like the Mandelbrot set contains countably many ``little 
Mandelbrot set'', it is experimentally observed (but not yet formally 
proved) that
the multicorn $\Md$ contains countably many ``little multicorns'', 
finitely many for each odd period $k\ge 3$ (these are combinatorial 
copies, not homeomorphic copies, for reasons we mentioned in the 
introduction); the period $n$ means that periods of hyperbolic 
components in the combinatorial copy are $n$ times the original 
periods. There is a natural map from the little multicorn onto $\Md$ 
that is given by an antiholomorphic version of the straightening 
theorem, but it is not necessarily continuous. Each little multicorn, 
say of period $k$, contains in turn countably many little multicorns, 
and all their periods are multiples of $k$. Under tuning (the inverse 
of straightening), the little multicorns thus form a semi-group (see 
Milnor \cite{MilnorHairiness}), and a ``prime'' multicorn is one that 
cannot be written as a composition of other small multicorns.
A map $p_c$ with an attracting or parabolic orbit of odd period $k$ 
is prime if the parameter $c$ is from the closure of the main 
hyperbolic component of a prime multicorn.
\end{remark}

Formally speaking, any map $p_c$ with a parabolic orbit of odd period 
$k\ge 3$ is clearly prime if the period $k$ is prime (it may or many 
not be prime otherwise). This establishes the existence of infinitely 
many OPPPP parabolics, using the existence of hyperbolic components of 
all periods.

Concerning the latter, define a sequence 
$s_{d,k}:=d^{k-1}-\sum_{m|k,\, m<n}s_{d,m}$ for each $d\ge 2$. Then 
$s_{d,k}$ is the number of hyperbolic components of period $k$ for 
the ``Multibrot sets'' of period $d$: each hyperbolic component of 
period $k$ has a center parameter that satisfies 
$((c^d+c)^d+c\dots+c)=0$, and dividing out solutions for periods $k$ 
strictly dividing $n$ we obtain the given recursive formula. It turns 
out that the number of hyperbolic components of the multicorns $\Md$ 
of period $k$  also equals $s_{d,k}$, except if $k$ is twice an odd 
number: in the latter case, the number of hyperbolic components 
equals $s_{d,k}+2s_{d,k/2}$ (Nakane and Schleicher, unpublished).

In order to reassure readers concerned that we might be talking about 
the empty set, here is a simple existence argument.
\begin{lemma}[Existence of Hyperbolic Components]
Every multicorn $\Md$ has hyperbolic components of all odd periods.
\end{lemma}
\begin{proof}
Let $k$ be an odd number and let $\phi\in\R/\Z$ be an angle with 
period $k$ under multiplication by $-d$ modulo $1$; i.e., 
$\phi=s/(d^k+1)$ for some $s\in\Z$. The parameter ray $R(\phi)$ at 
angle $\phi$ is defined as the set of parameters $c\in\C\sm\Md$ for 
which the critical value is on the dynamic ray at angle $\phi$ (and 
escapes to $\infty$). In \cite{Nakane1} it was shown that $\Md$ is 
connected, and in particular that $R(\phi)$ is a curve in 
$\C\sm\Md$ that converges to $\infty$ in one direction, and that 
accumulates at $\partial\Md$ in the other direction. Let 
$c\in\partial\Md$ be any accumulation point of $R(\phi)$; note that 
we do \emph{not} claim that $R(\phi)$ has a well-defined limit 
point in $\partial\Md$ (it usually will not), but its accumulation 
set is non-empty.

In the dynamics of $p_c$, the filled-in Julia set is connected, and 
the dynamic ray at angle $\phi$ lands at a periodic point that is 
repelling or parabolic. If the landing point is repelling, then 
stability under small perturbations assures that for parameters $c'$ 
near $c$, the dynamic ray at angle $\phi$ lands at a repelling 
periodic point, and ray and landing point depend continuously on 
$c'$. But since the critical value has positive distance from ray and 
endpoint, this will remain so under perturbations, and this is a 
contradiction (compare \cite[Lemma~B.1]{GoldbergMilnor}).

Therefore, for $p_c$, the dynamic ray at angle $\phi$ lands at a 
parabolic periodic point. Let $k'$ be the period of the parabolic 
orbit. Since $k'$ is odd, the rays landing at this orbit must have 
period $k'$ or $2k'$ \cite[Lemma~3.1]{Multicorns1}. This implies 
$k=k'$, so $c$ is on the boundary of a hyperbolic component of period 
$k$ (Lemma~\ref{Lem:ParabolicsOnBoundary}).
\end{proof}

\begin{lemma}[Analytic Arc Only for Real Parameters]
\label{Lem:AnalyticArcIsReal}
Suppose the filled-in Julia set of an OPPPP parabolic map $p_c$ 
contains a simple analytic arc that connects two bounded Fatou 
components. If the critical value has Ecalle height zero, then $p_c$ 
is conformally conjugate to a real map $p_{c'}$ (i.e., $c'\in\R$).
\end{lemma} 
\begin{proof}
Let $k$ be the period of the parabolic orbit, and let $z_1$ be the 
characteristic point on this orbit. Since the parabolic orbit is simple, any two bounded Fatou components have disjoint closures, so the analytic arc must traverse infinitely many bounded Fatou components.
Iterating the analytic arc 
forward finitely many times and cutting at the critical point if 
necessary, we obtain a simple analytic arc connecting $z_1$ to some 
other bounded Fatou component that intersects the parabolic tree. 
Truncate if necessary so that the arc does not meet any branch point 
of the parabolic tree, nor any bounded Fatou component that contains 
a branch point, but so that it still connects $z_1$ to some other 
bounded Fatou component, and so that all Fatou components that this 
arc intersects take more than $k$ iterations to reach the critical 
value Fatou component. Call this piece of analytic arc $J_1$. Then 
$p_c^{\circ k}\colon J_1\to p_c^{\circ k}(J_1)=:J_2$ is an analytic 
diffeomorphism between simple analytic arcs.

The arcs $J_1$ and $J_2$ are parts of the parabolic tree, except for 
homeomorphisms within bounded Fatou components (so they are part of a 
loose parabolic tree). They both start at $z_1$, which is not a branch point of the parabolic tree, so they must coincide at a Cantor set of points in the Julia set. As analytic arcs, they must thus coincide (except for truncation). It follows that one of the two arcs $J_1$ and 
$J_2$ is a sub-arc of the other. 
If $J_2\subset J_1$, then $p_c$ is 
not prime, so $J_2\supset J_1$ and hence $J_{n+1}:=p_c^{\circ 
k}(J_n)\supset J_n$ for all $n$. Again by definition of being prime, 
there is some $N$ so that $J_N$ covers the entire parabolic tree.

As long as $p_c^{\circ nk}\colon J_1\to J_{n+1}$ is a homeomorphism, 
the image is a simple analytic arc. If during the iteration, the critical 
point is covered, the $p_c$-image will contain the critical value, 
but this cannot introduce any branching: suppose $J=p_c^{\circ m}(J_1)$ is a simple analytic arc that contains the critical point $0$ in the interior and let $J'$ and $J''$ be the components of $J\sm\{0\}$. Then $p_c(J')$ and $p_c(J'')$ both start at the critical value and have $z_1$ as an interior point, so as above they must coincide in a neighborhood of $z_1$; hence $p_c(J')\cup p_c(J'')=p_c(J)$ is again a simple analytic arc. Therefore, all $J_n$ are simple analytic arcs, and the same holds for the parabolic tree, which equals $J_N$. The parabolic tree thus is unbranched.


Now we claim that $p_c$ is conformally conjugate to its complex 
conjugate $p_{\ovl c}$ (they are obviously conjugate by an \emph{anti}conformal homeomorphism, but we want a conformal conjugation). 
The condition of Ecalle height zero implies that $p_c^{\circ k}$ and $p_{\cbar}^{\circ k}$ are conformally conjugate on their incoming Ecalle cylinders, respecting the critical orbits. This conjugation can be pulled back to the incoming petal of the parabolic orbit and thus to their periodic Fatou components. It follows from local connectivity that this conformal conjugation on each individual Fatou component extends homeomorphically to the closure of the component.

The next step is to extend this conjugation homeomorphically to the filled-in Julia set, again using local connectivity. This is possible because the parabolic trees are unbranched, so their combinatorial structure is unaffected by complex conjugation. 

Finally, we extend the conjugation to the basin of infinity, using the Riemann map between the basins of infinity so that $\infty$ is fixed. There are $d+1$ choices for this conjugation near $\infty$, and one of them maps the rays landing at the parabolic orbit to the rays landing at the parabolic orbit (this is possible because we already know that the dynamics on the Julia set is conjugate); for this map, the extension to the boundary coincides with the conjugation on the Julia set we already have.

This way, we obtain a topological conjugation $h\colon\C\to\C$ between $p_c$ and $p_{\cbar}$ that is conformal away from the Julia set. If we know that the Julia set is holomorphically removable, then we have a conformal conjugation on $\C$. This fact can be established directly without too much effort. Consider the equipotential $E$ of $p_c$ at some positive potential, and let $E_1,\dots,E_k$ be piecewise analytic simple closed curves, one in each bounded periodic Fatou component of $p_c$, that surround the postcritical set in their Fatou components and that intersect the boundary of their Fatou components in one point, which is on the parabolic orbit. Let $V_0$ be the domain bounded on the outside by $E$ and on the inside by the $E_i$. Then there is a quasiconformal homeomorphism $h_0\colon \C\to\C$ with $h_0=h$ on $\C\sm V$ (i.e., the homeomorphism $h$ is modified on $V_0$ so as to become quasiconformal, possibly giving up on the condition that $h_0$ is a conjugation on $V_0$). 

Now construct a sequence of quasiconformal homeomorphisms $h_n\colon\C\to\C$ as a sequence of pull-backs, satisfying $p_{\cbar}\circ h_{n+1}=h_n\circ p_c$: the construction assures that this is possible, and all $h_n$ satisfy the same bounds on the quasiconformal dilatation as $h_0$. Moreover, the support of the quasiconformal dilatation shrinks to the Julia set, which has measure zero. By compactness of the space of quasiconformal maps with given dilatation, the $h_n$ converge to a conformal conjugation between $p_c$ and $p_{\cbar}$. This limiting conjugation must coincide with $h$ on the Fatou set, so the Julia set is holomorphically removable as claimed.

Finally, since $p_c$ and $p_{\ovl c}$ are conformally conjugate, we have 
$c=\zeta^s \ovl c$ where $\zeta$ is a complex $d+1$-st root of unity 
and $s\in\Z$, so writing $c=re^{2\pi i\phi}$ it follows that 
$\phi=-\phi+s/(d+1)$ or $\phi\in \Z/2(d+1)$. Since $p_c$ is 
conformally conjugate to $p_{c'}$ with $c'=c\zeta^{s'}$ for 
$s'\in\Z$, we may add $s'/(d+1)$ to $\phi$, so $\phi\in\{0,1/2\}$, 
and this means that $p_c$ is conformally conjugate to a real map.
\end{proof}

\begin{remark}
From the statement that the parabolic tree is unbranched, there is an alternative argument that $p_c$ is conformally conjugate to $p_{\ovl c}$.

We can modify the parabolic tree topologically into a superattracting tree. The map on this tree extends to a postcritically finite orientation-reversing branched mapping whose combinatorial equivalence class is well defined, and that is obviously combinatorially equivalent to its complex conjugate.

We can then apply Thurston's theorem (in fact, just Thurston rigidity) to claim that the corresponding superattracting antipolynomial $p_{c'}$ is conformally conjugate on $\C$  to its complex conjugate. This says that up to conjugacy, the tricorn with $p_{c'}$ at its center can be taken to be on the real axis, hence also the point on the principal boundary arc at Ecalle height $0$.

We give the more elementary but somewhat tedious argument above to avoid Thurston's theorem for orientation-reversing branched maps, though the result is true and requires almost no modifications in the proof. Similarly, one can use the methods of ``Posdronasvili'' \cite{Orsay} or of Poirier \cite{Poirier} to prove that any two postcritically finite polynomials with unbranched Hubbard trees having identical combinatorics are conformally conjugate.

\end{remark}

\begin{lemma}[Approximating Ray Pairs]
\label{Lem:ApproxRayPairs}
Every OPPPP parabolic map $p_c$ with Ecalle height zero is either 
conformally conjugate to a map $p_{c'}$ with real parameter $c'$, or 
the characteristic parabolic point $z_1$ is the limit of repelling 
preperiodic points $w_n$ and $\tilde w_n$ on the parabolic tree so 
that all $w_n$ have Ecalle heights $h>0$ and all $\tilde w_n$ have 
Ecalle heights $-h$, and with the following property: if $\phi$ and 
$\phi'$ are the external angles of the dynamic rays landing at $z_1$, 
then dynamic rays at angles $\phi_n$ and $\phi'_n$ land at $w_n$ so 
that $\phi_n\to\phi$ and $\phi'_n\to\phi'$; similarly, dynamic rays 
at angles $\tilde \phi_n$ and $\tilde \phi'_n$ land at $\tilde w_n$ 
with $\tilde\phi_n\to\phi$ and $\tilde\phi'_n\to\phi'$.
\end{lemma}
\begin{proof}
Repelling preperiodic points are dense on the Cantor set where the 
parabolic tree intersects the Julia set (for instance by the 
condition of being prime), so choose one such point, say $w_0$, in 
the repelling petal of $z_1$ (repelling periodic points must 
accumulate at $z_1$ and cannot do this within the attracting petal; 
and some neighborhood of $z_1$ is covered by the union of attracting 
and repelling petals). If all repelling periodic points on the 
parabolic tree and near $z_1$ have Ecalle height $0$, then the 
parabolic tree must intersect the Julia set entirely at Ecalle height 
zero, and then one can construct an analytic arc that satisfies the 
hypotheses of Lemma~\ref{Lem:AnalyticArcIsReal}, so $p_c$ is 
conformally conjugate to a real map.

If this is not the case, then we may assume that $w_0$ has non-zero 
Ecalle height $h$; to fix ideas, say $h>0$. Construct a sequence 
$(w_n)$ so that $w_{n+1}:=p_c^{\circ (-2k)}(w_n)$, choosing a local 
branch that fixes $z_1$ and so that all $w_n$ are in the repelling 
petal of $z_1$; hence $w_n\to z_1$ as $k\to\infty$. All $w_n$ have 
the same Ecalle height $h$.

As $w_0$ is on the parabolic tree, which is invariant, it follows 
that $w_0$ is accessible from outside of the filled Julia set on both 
sides of the tree, so $w_0$ is the landing point of (at least) two 
dynamic rays, ``above'' and ``below'' the tree. If $\phi_n$ and 
$\phi'_n$ are the corresponding angles of rays landing at $w_n$, then it follows that these 
sequences of angles converge to angles of rays landing at $z_1$ on 
both sides of the tree, and the claim follows.

Now let $w'_n:=p_c^{\circ k}(w_n)$; then $w'_n\to z_1$ and all these 
points have Ecalle heights $-h$. The rays landing at $z_1$ have 
angles $\phi$ and $\phi'$ and their period is $2k$, so $p_c^{\circ 
k}$ permutes these and the claim about $w'_n$ and its rays follows.
\end{proof}

\section{Non-Pathwise Connectivity}
\label{Sec:Non-Pathwise}

We denote the dynamic ray at angle $\phi$ for the map $p_c$ by $R_c(\phi)$, and as before the parameter ray at angle $\phi$ by $R(\phi)$.


\begin{theorem}[Rays Approximating at OPPPP Arc]
\label{Thm:RaysApproximatingArc}
Let $\mathcal A$ be a prime parabolic root arc of odd period $k\ge 3$ 
that does not intersect the real axis or its images by a symmetry 
rotation of $\M_d$, and let $c\in\mathcal A$ be the parameter with 
Ecalle height zero. Let $\phi$ and $\phi'$ be the characteristic angles of the parabolic orbit for parameters $c\in\mathcal A$. 
Then there is a sub-arc $\mathcal A_\tau\subset\mathcal A$ of positive length and there are angles $\tilde\phi_n\to\tilde\phi$ and $\phi'_n\to\phi'$ so that $\mathcal A_\tau$ is contained in the limit of the parameter rays  $R(\tilde\phi_n)$, and also of $R(\phi'_n)$ (this is the limit of the sequence of rays, not necessarily of individual rays:  each $c\in\mathcal A_\tau$ is the limit of a sequence of points on the parameter rays $R(\tilde\phi_n)$, and of another sequence on $R(\phi'_n)$.)
\end{theorem}
\begin{proof}
In the dynamics of $p_c$, let $z_1$ be the characteristic parabolic 
point. The angles $\phi$ and $\phi'$ have period $2k$, so the rays $R_c(\phi)$ and $R_c(\phi')$ are interchanged by the first return map of $z_1$. In the outgoing Ecalle cylinders at $z_1$ of the holomorphic 
map $p_c^{\circ 2k}$, the rays $R_c(\phi)$ and $R_c(\phi')$ project to disjoint simple closed curves, not necessarily at constant Ecalle 
heights, but it makes sense to say which of the two rays has greater 
Ecalle heights (removing the projection of one ray from the Ecalle 
cylinder, the other ray is in the component with arbitrarily large 
positive or negative Ecalle heights). Without loss of generality, 
suppose that $R_c(\phi)$ has greater heights than $R_c(\phi')$. Let $h^+$ be the maximum of Ecalle heights of $\phi$, and $h^-$ be the minimum of Ecalle heights of $\phi'$; since $p_c^{\circ k}$ interchanges $R_c(\phi)$ and $R_c(\phi')$, we have $h^+=-h^->0$.

Consider the sequences of repelling preperiodic points $w_n$ and 
$\tilde w_n$ converging to $z_1$ as provided by 
Lemma~\ref{Lem:ApproxRayPairs}, and let $h>0$ and $-h$ be their 
Ecalle heights. Then clearly $h<h^+$ (the points $w_k$ are in the part of the Ecalle cylinder bounded by the rays $R_c(\phi)$ and $R_c(\phi')$). 
 The rays $R_c(\tilde \phi_n)$ terminate at the points $\tilde w_n$ with Ecalle heights $-h$, while they all project to the same ray in the Ecalle cylinder, in which they spiral upwards and converge towards the projection of the ray at angle $\phi$. Therefore, for any compact subinterval of $(-h,h)$, the rays $R_c(\tilde\phi_n)$ have Ecalle heights in this entire compact interval. Similarly, the rays $R_c(\phi'_n)$ terminate at the $w_n$ with Ecalle height $h$ and also have Ecalle heights within any compact subinterval of $(h^-,h)$; see Figure~\ref{Fig:RaysOverlappingEcalleHeights}.

Now let $c_t\in\mathcal A$ be the parameter where the critical 
value has Ecalle height $t\in\R$ (see 
Proposition~\ref{Prop:ParabolicArc}). The points $w_k$ depend 
real-analytically on $t$ (like the entire Julia set); let $h(t)$ be 
their Ecalle heights; these too depend analytically on $t$. 
Therefore, there is a $\tau\in(0,h)$ so that $h(t)>t$ for all 
$t\in(-\tau,\tau)$. Choose $\eps\in(0,\tau)$.
Let $\mathcal A_\tau\subset\mathcal A$ be the sub-arc with Ecalle heights in $(-\tau+\eps,\tau-\eps)$. 

To transfer these dynamic rays from the Ecalle cylinders to parameter space, we employ Proposition~\ref{Prop:PerturbedFatouCoords}. Choose any smooth path $c\colon[0,\delta]\to\C$ with $c(0)=c_t\in \mathcal A_\tau$ but so that except for $c(0)$ the path avoids closures of hyperbolic components of period $k$, and so that the path is transverse to $\mathcal A$ at $c_t$. 

In the outgoing cylinder of $c(0)=c_t\in\mathcal A$, all $R_c(\tilde\phi_n)$ traverse Ecalle heights in $[-h+\eps/2,h-\eps/2]$.  Since each ray $R_c(\tilde \phi_n)$ and its landing point depend uniformly continuously on $c$, and since the projection into Ecalle cylinders is also continuous, there is a $\delta_\eps>0$ so that for all $c(s)$ with $s<\delta$ the projection of the rays $R_c(\tilde\phi_n)$ into the Ecalle cylinders traverses heights $[-h+\eps,h-\eps]$, while the phase is uniformly continuous in $s$. 

For $s\in[0,\delta]$, let $z(s)$ be the critical value. For $s>0$, the critical orbit ``transits'' from the incoming Ecalle cylinder to the outgoing cylinder; as $s\searrow 0$, the image of the critical orbit in the outgoing Ecalle cylinder has Ecalle height tending to $t\in(-\tau+\eps,\tau-\eps)\subset(-h+\eps,h-\eps)$, while the phase tends to infinity. 
Therefore, there are $s\in(0,\delta_\eps)$ arbitrarily close to $0$ at which the critical value, projected into the incoming cylinder and sent by the transfer map to the outgoing cylinder, lands on the projection of the rays $R_c(\tilde \phi_n)$. But in the dynamics of $p_{c(s)}$, this means that the critical value is on one of the dynamic rays $R_c(\tilde\phi_n)$, so $c(s)$ is on the parameter ray $R(\tilde\phi_n)$. 

The analogous statement holds for $\phi'_n$.
\end{proof}

\begin{figure}
\framebox{\Includegraphics{height=80mm,trim=20 0 20 0,angle=90}{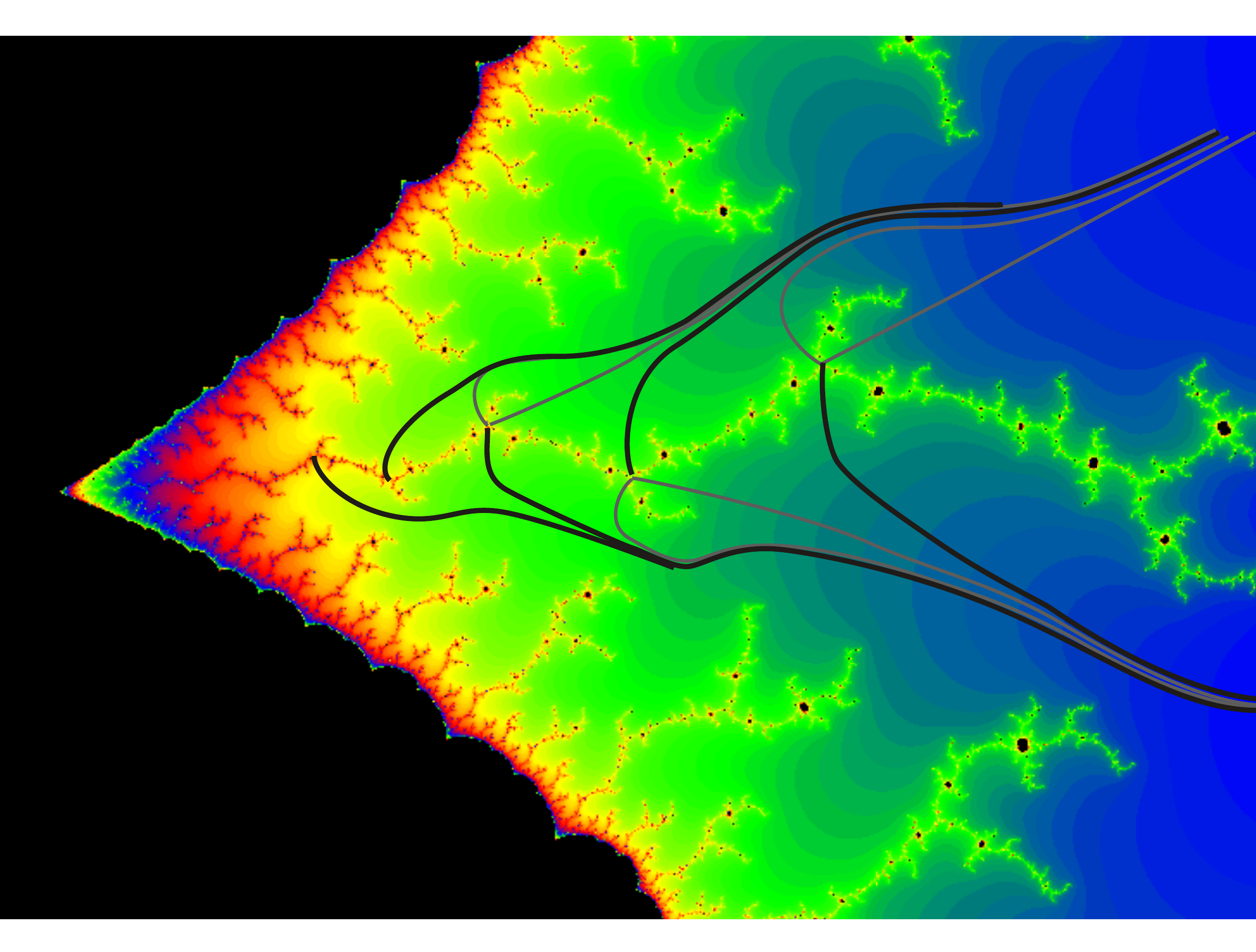}}
\framebox{\Includegraphics{width=80mm}{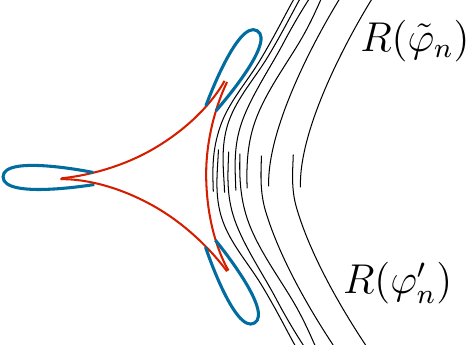}}
\framebox{\Includegraphics{width=80mm,angle=180,trim=125 230 90 180,clip}{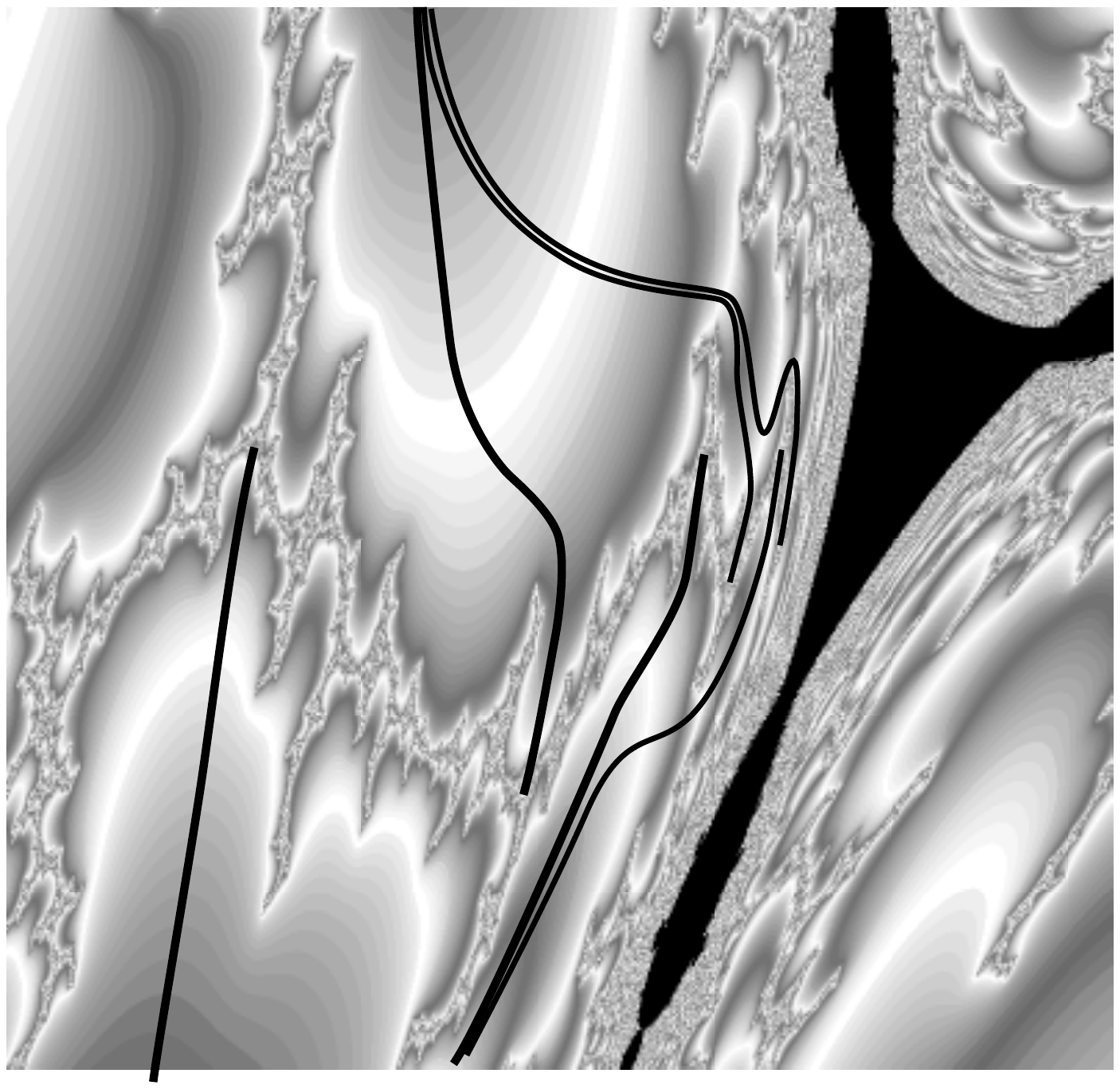}}
\caption{Loss of pathwise connectivity because of approximating 
overlapping parameter rays. Top: approximating preperiodic dynamic rays in the dynamic plane with a parabolic orbit. Only the rays drawn by heavy lines are used in the argument below; other rays landing at the same points are drawn in grey. Middle: symbolic sketch of the situation in the parameter space. Bottom: Actual parameter rays accumulating in the same pattern, producing a double-comb-like structure.
}
\label{Fig:RaysOverlappingEcalleHeights}
\end{figure}


Now the proof of our main result is simple.

\begin{theorem}[Multicorns Are Not Path-Connected]
\label{Thm:LossPathwise}
For each $d\ge 2$, the Multicorn $\Md$ is not path connected. 
\end{theorem}
\begin{proof}
Let $W$ be any hyperbolic component of odd period not intersecting the real axis, let $\mathcal A\subset W$ be the parabolic root arc and suppose it is prime. By Theorem~\ref{Thm:RaysApproximatingArc}, there is a sub-arc $\mathcal A_\tau$ of positive length and there are two sequences of angles $\tilde\phi_n$ and $\phi'_n$ converging to limits $\tilde\phi\neq\phi'$ so that the set
\[
\bigcup_n R(\tilde\phi_n) \cup \bigcup_n R(\phi'_n)\cup\mathcal A_\tau
\]
disconnects $\C$ into at least $2$ path-components. If the angles are oriented so that $\phi<\phi'$, then $W$ is in a different component from $R(0)$ and any hyperbolic component $W'$ in the limit set of the angle $R(1/(2^n\pm 1))$ for sufficiently large $n$: any path connecting $W$ to $W'$ must accumulate at all points in $\mathcal A_\tau$, and this is impossible.
\end{proof}

\begin{remark}
We believe that the only hyperbolic components for which the
umbilical cord lands are on the real axis, or symmetric to such 
components by a rotational symmetry of $\Md$. For individual 
components, this can be verified numerically: all one needs to know 
is that the parabolic tree does not contain analytic arcs, even in 
the non-prime situation; for this it is good enough to know that the subtree of renormalizable points 
contains periodic points with non-real multipliers.

\end{remark}

\section{Further Results}
\label{Sec:FinalSection}

The following proposition and its proof are inspired by more general results 
due to Bergweiler \cite{Bergweiler} as well as Buff and Epstein 
\cite{BE}.


\begin{theorem}[No Bifurcation at Ecalle Height Zero]
\label{Thm:EcalleHeightZeroNoBifurcation}
On every parabolic arc of period $k$, the point with Ecalle height zero has a neighborhood (along the arc) that does not intersect the boundary of a hyperbolic component of period $2k$.
\end{theorem}
\begin{proof}
Suppose $p_c$ is the center point of the parabolic arc (at Ecalle 
height $0$). We will now 
discuss the local dynamics of the holomorphic first return map, i.e., 
the $2k$-th iterate of $p_c$. The upper and lower endpoints of the 
Ecalle cylinders correspond to fixed points of $p_c^{\circ 2k}$; but 
they are interchanged by $p_c^{\circ k}$, so they are simultaneously 
attracting or repelling with complex conjugate multipliers.

Points in the outgoing petal at sufficiently large positive Ecalle 
heights will return to the incoming petal; this map is called the 
``horn map''. This induces a conformal map from the upper end of the 
outgoing cylinder to the incoming cylinder which, by the Koebe 
compactness theorem, is close to a translation by a complex constant 
(writing the cylinders as $\Cstar$, it is close to multiplication by 
a constant). Let $\eta$ be the imaginary part of the translation 
constant.
Therefore, for every $\eps>0$ there is an $H>0$ so that points at 
Ecalle heights $h_o>H$ in the outgoing cylinder return 
to points in the incoming cylinder with height $h_i$ so that 
$|h_i-h_o+\eta|<\eps$, i.e.\ $h_i\in(h_o-\eta-\eps,h_o-\eta+\eps)$.

Cut the outgoing and incoming Ecalle cylinders at the equators into 
upper and lower half-cylinders and label the upper halves $C_o$ and 
$C_i$, where $i$ and $o$ stand for ``incoming'' and ``outgoing'' (the 
whole discussion can be done analogously in the lower halves, with 
negative Ecalle heights, just as well; the antiholomorphic iteration 
assures that they are completely symmetric).
Let $C'_o\subset C_o$ be the restriction to the parabolic 
basin. Let $f\colon C'_o\to C_o$ be a conformal isomorphism; it 
is unique up to addition of a real constant (a phase). There is a 
number $\delta\in\C$ so that asymptotically near the end, 
$f(z)=z-i\delta$. By adjusting the freedom in $f$, we can make 
$\delta$ purely real. The Schwarz Lemma, together with the fact that $C_o\sm C'_o$ contains open sets in the basin of infinity, implies  $\delta>0$. This 
number is called the \emph{Gr\"otzsch defect} of the outgoing 
cylinder (with respect to the parabolic basin).

For $h>0$, let $C_i(h)$ and $ C'_o(h)$ be $C_i$ and $C'_o$ 
restricted to  Ecalle heights in $(0,h)$. Then $h=\bmod(C'_o(h))$ and $\delta$ 
can be viewed as the limit, as $h\to \infty$, of $h-\bmod(C'_o(h))$.
The Gr\"otzsch inequality implies that $\bmod(C'_o(h))\le h-\delta$ for all $h$.

Choose $\eps\in(0,\delta)$ and $H$ depending on $\eps$ as above.
Since by hypothesis the critical orbit is at Ecalle 
height $0$, one can pull back $C_i(H)$ conformally into $C_o$; it 
must land within $ C'_o(H+\eta+\eps)$ (it must have Ecalle height 
less than $H+\eta+\eps$ and it must be in the part within the 
attracting basin). Since $\bmod(C_i(H))=H$, while $\bmod( 
C_o'(H+\eta+\eps))\le H+\eta+\eps-\delta$, this implies 
$\eta+\eps-\delta\ge 0$ and thus $\eta>0$.

By Proposition~\ref{Prop:PerturbedFatouCoords}, the outgoing and 
incoming cylinders exist after small perturbations of the parameter 
outside of its hyperbolic component, the Ecalle heights depend 
continuously on the perturbation, and the same holds for the ``horn 
maps'' from the ends of the outgoing into the incoming cylinders.

If $\eps<|\eta|$, then $\eta<0$ means that points with great Ecalle 
heights the outgoing cylinder return into the incoming cylinder at 
greater heights.
For such sufficiently small perturbations, $\eta<0$ thus implies that 
points near the end of the cylinder will, after perturbation, 
converge to the end of the cylinder; the endpoint of the cylinder 
thus becomes an attracting fixed point. Similarly, if $\eta>0$, then 
the endpoints become repelling. Since these are the periodic points 
that bifurcate from the period $k$ orbit, this shows that parameters 
$c$ with $\eta>0$ are not on the boundary of a period $2k$ hyperbolic 
component, and this is the case when the Ecalle height $h$ is zero or 
sufficiently close to zero.
\end{proof}

\begin{remark}
This result can be strengthened in various ways. One can give an 
explicit lower bound on the Gr\"otzsch defect $\delta$: the basin at 
infinity alone occupies an annulus of modulus at least $1/(2k\log 
d)$, so $\delta>1/(2k\log d)$ (compare \cite{BE}, Theorem~B and 
especially the first half of the proof). Moreover, one can deduce an 
inequality between the Ecalle height of the critical orbit and the fixed 
point index: if the critical value has Ecalle height $h$, one can 
estimate the conformal modulus of the largest embedded annulus in 
$C_i(H)$ that separates the critical value from the upper boundary 
(this is a classical extremal length estimate; the modulus is 
$H-|h|+o(1)$), and this gives a correspondingly greater upper bound 
on the fixed point index (after all, we know that for large Ecalle 
heights $h$,  the fixed point index must become greater than $1$). 
Combining both facts, this implies a definite interval of Ecalle 
heights around $0$, depending only on $d$, for which the parabolic 
arc does not meet bifurcating components.
\end{remark}

\begin{theorem}[Decorations Along Parabolic Arc]
\label{Thm:ArcDecorations}
Every parabolic arc on a hyperbolic component of odd period $k$ has 
Ecalle heights $h_1,h_2,h_3, h'_1,h'_2,h'_3\in\R$ so that 
$h_3>0>h'_3$, $h_3>h_2$, $h'_2>h'_3$ satisfying the following 
properties:
\begin{itemize}
\item
the sub-arc with Ecalle heights $h>h_3$ is an arc of bifurcation to a 
component of period $2k$; and also for Ecalle heights $h<h'_3$;
\item
the sub-arc with Ecalle heights $h\in (h_2,h_3)$ is the limit of 
decorations (attached to the period $2k$ components bifurcating for 
large positive Ecalle heights); and also for Ecalle heights 
$h\in(h'_3,h'_2)$;
\item
if the arc is a root arc, then the sub-arc with Ecalle heights 
$h\in[h'_1,h_1]$ is the limit of the ``umbilical cord''.
\end{itemize}
\end{theorem}
Note that by Theorem~\ref{Thm:EcalleHeightZeroNoBifurcation} we have 
$h_3> 0>h'_3$ (the two hyperbolic components near the end of the 
parabolic arc are disjoint), but we do not know whether always 
$h_2>0>h'_2$ (if $h_2<h'_2$, this would mean that the decorations 
from the period $2k$ components at both ends of the arc would 
overlap; this would imply $h_1\ge |h'_2|$). We clearly have 
$h_3>\max\{h_2,|h'_2|, h_1\}$. Loss of pathwise connectivity of the 
umbilical cord occurs whenever $h_1>0$.

If $h_1<h_2$, then we have an ``open beach'' where the boundary of 
the multicorn locally equals just the parabolic arc without further 
decorations. We do not know whether there are infinitely many 
parabolic arcs for which this occurs.

\begin{figure}[htbp]
\framebox{\Includegraphics{height=65mm,trim=0 0 000 0}{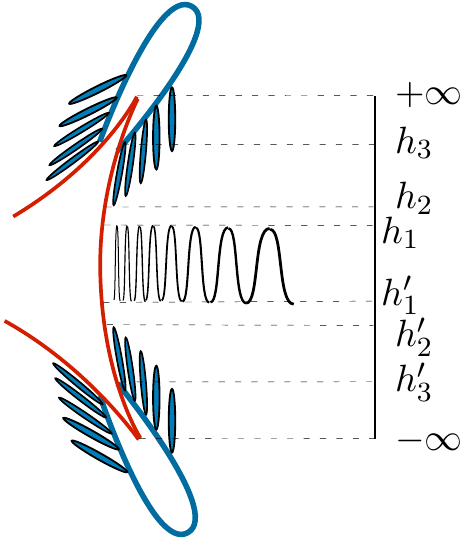}}
\framebox{\Includegraphics{height=65mm,trim=0 000 0 0,clip}{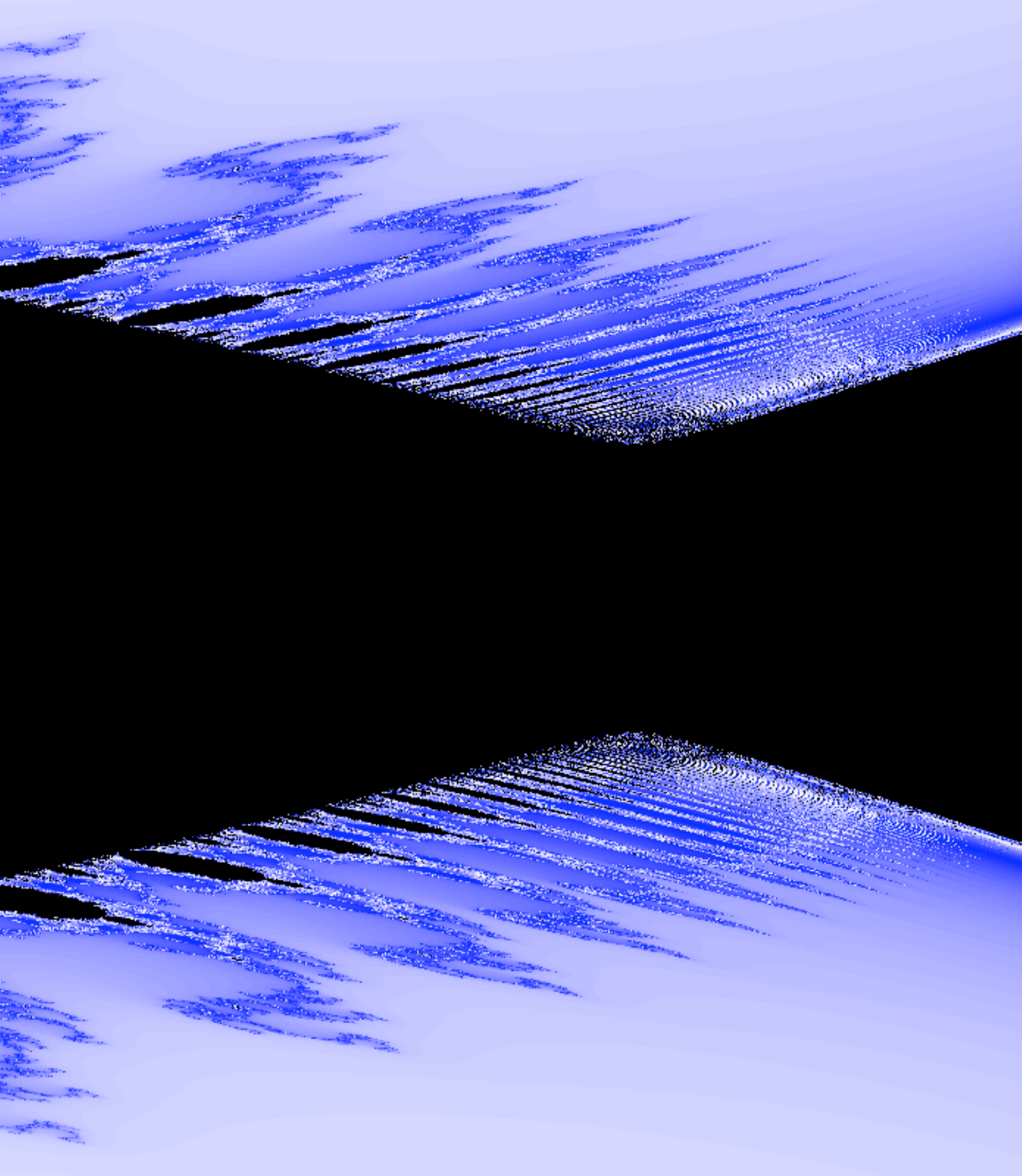}}
\caption{Illustration of Theorem~\ref{Thm:ArcDecorations}. Left: schematic illustration of the decorations along parabolic arcs, together with their threshold heights. Right: 
decorations at a period $2$ component that accumulate along arcs on the boundary of the period $1$ component.}
\end{figure}

\begin{proof}[Sketch]
For sufficiently large Ecalle heights, the parabolic arc is on the
locus of bifurcation from odd period $k$ to period $2k$
(Corollary~\ref{Cor:BifArcs}). Let $h_3\in\R$ be the infimum of such
Ecalle heights, and consider $p_c$ for $c$ on this parabolic arc with
Ecalle height $h_3$. The boundary of the Fatou component, projected
into the outgoing Ecalle cylinder, is not a geodesic in this cylinder
(because this boundary is not an analytic curve). It will thus
project into the cylinder at an interval $(h',h'')$ of Ecalle heights
with $h''$ strictly greater than $h'$. In fact, we have $h''=h_3$.

For parameters $c(h)$ with $h$ slightly less than $h_3$, the lower
Ecalle height $h'$ depends on $h$, but there is an interval
$(h_2,h_3)$ when $h'(h)<h$. For these, both escaping points and the
Julia set intersect the outgoing Ecalle cylinder at Ecalle height in
a neighborhood of $h$, so the parameter $c(h)$ can be approximated by
parameters inside and outside of $\Md$.
\end{proof}

\begin{remark}
Each parabolic arc contains the accumulation set of one or two 
periodic parameter rays (two rays for root arcs, one for co-root 
arcs). By symmetry, the rays accumulating at the boundary arcs of the 
period $1$ hyperbolic component actually land. We believe that all 
the other periodic parameter rays do not land, at least those that 
accumulate at root arcs. Instead, we believe that they 
accumulate at a sub-arc of positive length.

The reason is as follows. For a parameter $c$ on a root arc of 
period $n\ge 3$, the parabolic periodic point is the landing point of 
$2$ dynamic rays. These form hyperbolic geodesics in the access that 
is on one side bounded by a periodic Fatou component, and on the 
other side by the parabolic tree, decorated by various structures of 
the Julia set. Even though the ray is an analytic curve, it would 
seem to require a miracle that the boundaries of the access at the 
two sides are symmetric enough so that the ray projects to an equator 
in the Ecalle cylinder. But if it does not project to an equator, but 
has varying Ecalle height instead, then these wiggles will transfer 
into parameter space to a ray that accumulates on the parabolic arc 
like a topologist's sine curve (see Figure~\ref{Fig:WigglyDynamicRays}).
\end{remark}

\begin{figure}
\hide{
\framebox{\Includegraphics{height=60mm,trim=0 0 0 30,clip}{EcalleCylinderJuliaRays2.pdf}}
}
\framebox{\Includegraphics{height=60mm,trim=0 0 0 30,clip}{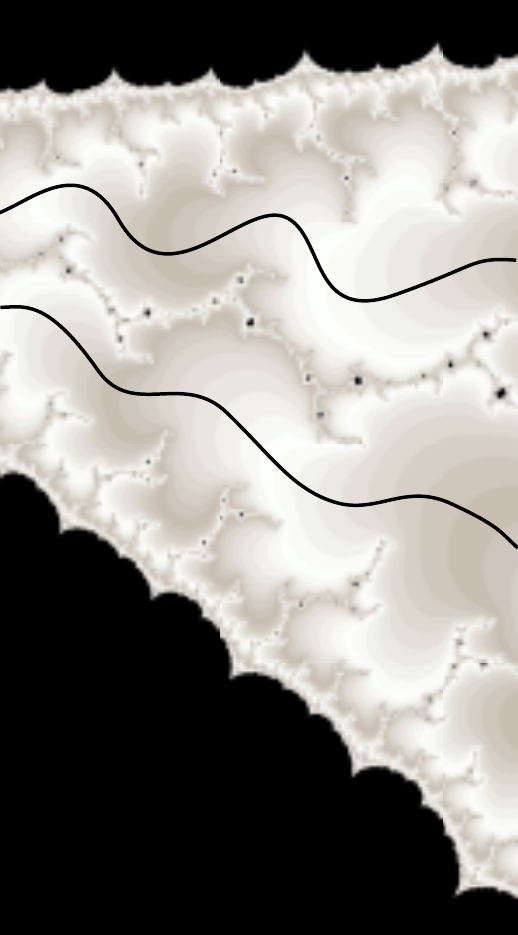}}

\caption{Heuristic argument why parameter rays should not land at parabolic root arcs but rather accumulate at a sub-arc of positive length. Sketch of the situation in the parabolic dynamics: near the top and bottom, there is the parabolic Fatou component (black, visible only in the bottom); in the middle there is the Julia set around the parabolic tree, and the two black curves are the two dynamic rays landing at the parabolic periodic point. They have no reason to have constant Ecalle height between two very different structures. While the parabolic basin is not stable under perturbations, the rays move continuously and keep their wiggles. }
\label{Fig:WigglyDynamicRays}
\end{figure}

\hide{
\begin{proposition}
The boundary of every period $2n$ component that bifurcates from a 
period $n$ component with $n$ odd intersects the boundary of the 
latter in two isolated points.
\end{proposition}
\begin{proof}[Sketch]
The intersection of the boundaries is on a parabolic arc at a point 
with fixed point index $1$.
\end{proof}
}

\begin{remark}
It is tempting to try to show for hyperbolic components of odd period 
$k$ that the bifurcating period $2k$ component has wiggly boundary 
near the parabolic arc, transferring the wiggly boundary of the Fatou 
component to parameter space (using the fact that these Fatou 
components do not have analytic boundary arcs). However, this Fatou 
component is not stable under perturbation away from the parabolic 
arc, and we do not obtain wiggles in parameter space: in fact, it 
follows from Theorem~\ref{Thm:BifurcationArcs} that the boundary of 
the period $2k$ component has a well-defined limit point on the 
parabolic arc: a simple parabolic with fixed point index $+1$, and 
those are isolated. What can be transferred into parameter space are 
the repelling periodic points on the boundary of the Fatou component, 
and the rays landing at them. These yield the decorations of the 
period $2k$ components that accumulate at parabolic arcs in a 
comb-like manner, as described in Theorem~\ref{Thm:ArcDecorations}.
\end{remark}


\end{document}